\documentclass[a4paper,11pt]{amsart}
\usepackage[utf8]{inputenc}
 
\usepackage{amsmath, amssymb}
 
\usepackage{setspace}
\usepackage[left=3cm, right=3cm, top=3cm, bottom=3cm]{geometry}                 

\usepackage{cancel}

\usepackage{graphicx} 
\usepackage{pgf,tikz}
\usetikzlibrary{arrows}
\usepackage{pdfsync}
\usepackage{mathrsfs}
\usepackage{hyperref} 
\usepackage{verbatim} 
\usepackage{epstopdf}
\usepackage{xargs}
\usepackage{todonotes}
\hypersetup{colorlinks,
	linkcolor={green!30!black},
	citecolor={blue!40!black},
	urlcolor={black}}
\usepackage[capitalise]{cleveref}

\usepackage[normalem]{ulem}
\usepackage[doi=false,isbn=false,url=false,maxbibnames=99]{biblatex}

\newcommandx{\jow}[2][1=]{\todo[linecolor=orange,backgroundcolor=orange!25,bordercolor=orange,#1]{#2}}

\newcommandx{\mateus}[2][1=]{\todo[linecolor=blue,backgroundcolor=blue!25,bordercolor=blue,#1]{#2}}

\def\R{\mathbb{R}}

\def\N{\mathbb{N}}
\def\Z{\mathbb{Z}}

\def\C{\mathbb{C}}

\renewcommand{\d}{\mathrm{d}}

\newcommand{\eps}{\varepsilon}

\newcommandx{\eq}{\approxeq}

\newtheorem{theorem}{Theorem}
\newtheorem{corollary}{Corollary}
\newtheorem{proposition}{Proposition}

\newtheorem{remark}[proposition]{Remark}
\newtheorem{lemma}[proposition]{Lemma}

\bibliography{references.bib}	

\numberwithin{equation}{section}
\numberwithin{proposition}{section}

\date{January 12, 2022}
\title[Uniqueness from zeros]{Uniqueness of solutions to nonlinear Schr\"odinger equations from their zeros}
\author{Christoph Kehle}
\address{Institute for Advanced Study, School of Mathematics, 1 Einstein Drive, Princeton NJ, USA and ETH Zürich, Institute for Theoretical Studies, Clausiusstrasse 47, 8092 Zürich}

\email{c.kehle@ias.edu}
\author{João P. G. Ramos}
\address{ETH Zürich, Department of Mathematics, Rämisstrasse 101, 8092 Zürich}
\email{joao.ramos@math.ethz.ch}
\begin{document}
	
\begin{abstract}

We show  novel types of uniqueness and rigidity results for Schrödinger equations in either the nonlinear case or in the presence of a  complex-valued potential. As our main result we obtain that the trivial solution $u=0$ is the only solution for which the assumptions $u(t=0)\vert_{D}=0, u(t=T)\vert_{D}=0$ hold, where $D\subset \R^d$ are certain subsets of codimension one. In particular, $D$ is \emph{discrete} for dimension $d=1$.

Our main theorem can be seen as a nonlinear analogue of discrete Fourier uniqueness pairs such as the celebrated Radchenko--Viazovska formula in  \cite{MR3949027}, and the uniqueness result of the second author and M. Sousa for powers of integers \cite{ramos2019fourier}. As an additional application, we deduce rigidity results for solutions to some semilinear elliptic equations from their zeros.

\end{abstract} 
\maketitle
 {\hypersetup{hidelinks}
\tableofcontents
	}

\section{Introduction} 

Let $f \in L^2(\R).$ The problem of determining necessary and sufficient conditions on sets $A,B$ so that $f|_A = \widehat{f}|_B = 0$ implies $f \equiv 0$ has been extensively studied in connection with the celebrated Heisenberg uncertainty principle. If a pair of sets $(A,B)$ possesses the property stated above, possibly under additional assumptions on the functions $f$, we call it a \emph{Fourier uniqueness pair}. 

For instance,  $(A,B) = (\R \setminus (-1/2,1/2), \Z)$ is seen, due to the Shannon--Whittaker interpolation formula for band-limited functions, to be a Fourier uniqueness pair. In the celebrated work \cite{MR3949027},   Radchenko and Viazovska recently proved that, restricted to the class of Schwartz functions, $(A,B) = (\sqrt{n},\sqrt{n})_{n\in \mathbb N}$ is a Fourier uniqueness pair. After the Radchenko--Viazovska breakthrough, several other results came about giving necessary and sufficient conditions on the pair $(A,B)$ so that the Fourier uniqueness property holds, for which we refer the reader to \cite{ramos2021perturbed,stoller2021fourier,Kul21} and the references therein.

While the Fourier transform is a rather distinguished and highly symmetric operator, it is now a  natural problem to explore the validity of such uniqueness results for more general operators   which are known to enjoy similar uncertainty properties. 
 Indeed, in the recent work \cite{goncalves2020note} by F.~Gon\c{c}alves and the second author, the first step connecting Fourier uniqueness pairs to other operators has been taken. There, the authors consider the \emph{free Schr\"odinger equation} 

\begin{align}
\label{eq:free-evolution}
\begin{cases}
i \partial_t u + \Delta u = 0 & \text{ in } \R \times \R;  \cr 
u(x,0) = u_0(x) & \text{ on } \R.  \cr 
\end{cases} 
\end{align}

Among other properties, the authors prove that, for certain discrete subsets of the $(x,t)-$plane, if a solution $u$ vanishes at the points of that set, and $u_0$ belongs to a fast-decaying class of functions, then $u \equiv 0.$ 

The result relies heavily on the explicit representation of the free evolution of \eqref{eq:free-evolution} given by $ u = e^{it\Delta } u_0 = \frac{1}{  \sqrt{4t \pi i} } e^{-|x|^2/4it} \ast u_0$ which allows to resemble several aspects of the previously mentioned rigidity results for the Fourier transform. An important question immediately raised is whether such uniqueness results extend for the evolution in the presence of a given \emph{potential} or  for \emph{nonlinear interactions}.  Indeed, already in the linear case of the Schrödinger equation with a  (possibly complex-valued) potential
\begin{equation}\label{eq:schrodinger-potential-0}
\begin{cases}
i \partial_t u + \Delta u + V u= 0 & \text{ in } \R \times [0,T];  \cr 
u(x,0) = u_0(x) & \text{ on } \R, \cr 
\end{cases} 
\end{equation}
the techniques used by Gon\c{c}alves and the second author \cite{goncalves2020note}, which mainly rely on relating the evolution of the free Schr\"odinger equation to suitable Fourier transforms, completely fall apart, and it seems that new techniques are needed then. 

The main purpose of this work is to provide first examples of \emph{discrete} uniqueness sets for the Schr\"odinger equation, either in the case of a potential or in the nonlinear case. Before we introduce our main  results in \cref{sec:main-results}, however, we summarize some previous progress on the question of unique continuation for the Schr\"odinger equation. 

Indeed, our starting point will be a result on unique continuation for the Schr\"odinger equation by Kenig, Ponce and Vega \cite{KPV03}, where the authors proved the following. Let \linebreak $u_1,u_2 \in C([0,T]; H^s(\R^d))),$ $ \, s \geq  \max\{d/2+,2\},$  
be solutions to 
\begin{equation}\label{eq:schrodinger-nonlinear-general} 
i	\partial_t u + \Delta u + F(u,\overline{u}) = 0 \text{ in } \R^d\times [0,T],
	\end{equation}
where $F \in C^{\lfloor s \rfloor + 1} (\C^2 \colon \C)$ satisfies $|\nabla F(u,\overline{u})| \lesssim (|u|^{p_1 - 1} + |u|^{p_2 - 1}), p_1, p_2 >1.$ If 
$$u_1(x,0) = u_2(x,0), \, \, u_1(x,T) = u_2(x,T), \, \, \forall x \not\in \Gamma+ y_0,$$
where $\Gamma$ is a convex cone contained in a half space, then $u_1 \equiv u_2.$ One of the key steps in their proof is a lemma---which we shall use in this paper; see Lemma \ref{lemma-KPV} for more details---which roughly states that exponential decay is preserved for times $t \in (0,T),$ as long as we have exponential decay of the solution at two times $t=0,T.$ 

After the Kenig--Ponce--Vega breakthrough, several other advances in the problem of determining sharper conditions on $u_1-u_2$ for uniqueness have been made, such as \cite{EKPVSurvey}, the recent survey \cite{linares2021unique} for non-local operators and the references therein. In particular, in the series of papers  \cite{EKPV06,EKPV08,EKPV10}, Escauriaza--Kenig--Ponce--Vega obtained a \emph{sharp form} of Hardy's uncertainty principle for the Schr\"odinger evolution: suppose $u \in C([0,T];L^2(\R^d))$ is a solution to \eqref{eq:schrodinger-potential-0}, and suppose additionally that there are two constants $\alpha,\beta > 0$ with $\frac{\alpha \beta}{4T} < 1$ such that 
\begin{equation}\label{eq:gaussian-decay-schrodinger}
\|e^{|x|^2/\beta^2} u(0)\|_2 + \|e^{|x|^2/\alpha^2} u(T)\|_2 <+\infty. 
\end{equation}
If the potential $V$ satisfies $\lim_{R \to \infty} \|V\|_{L^1([0,T];L^{\infty}(\R^d \setminus B_R))} = 0,$ then one has $u \equiv 0.$ They also proved that the conditions on $\alpha,\beta$ are \emph{sharp}, as in the case of $\alpha \beta = 4T$ there is a potential satisfying the condition above and a non-zero solution to the Schr\"odinger equation \eqref{eq:schrodinger-potential-0} satisfying \eqref{eq:gaussian-decay-schrodinger}.  

As a consequence of their results for the potential case \eqref{eq:schrodinger-potential-0}, they obtain that if two solutions $u_1,u_2 \in C([0,T];H^k(\R^d)), k > d/2,$ of \eqref{eq:schrodinger-nonlinear-general}, with $F \in C^k, \, F(0) = \partial_u F(0) = \partial_{\overline{u}} F(0) = 0,$ satisfy 
$$e^{|x|^2/\beta^2} (u_1(0) - u_2(0)), e^{|x|^2/\alpha^2} (u_1(T) - u_2(T)) \in L^2,$$
where $\alpha \beta < 4T,$ then $u_1 \equiv u_2.$  \\

\section{Main results} 
\label{sec:main-results}
As highlighted before, the results above deal with decay properties of solutions to Schr\"odinger equations on \emph{large} spatial sets, in particular sets with zero codimension. Our main results are concerned with knowledge of the solution on \emph{small} spatial sets with codimension one. In particular, in the one-dimensional case, we obtain rigidity already from zeros (or decay) of the solution imposed on a \emph{discrete} set.

We first consider the one dimensional case, where we show a robust rigidity result for the Schrödinger equation in the presence of a complex potential (\cref{thm:first-weak-thm}) and for general nonlinearities (\cref{thm:uniqueness-one-dim}). The theorem, however, only applies in dimension $d=1$ and we have to assume rapidly accumulating zeros (or decay) at infinity. This is the content of \cref{sec:one-d-result-complex-pot}.
Then, in  \cref{sec:main-result-d>=1}, we state our main rigidity results \cref{thm:rigidity} ($d=1$) and \cref{thm:rigidity-high-dim} ($d\geq2$)  for the cubic NLS assuming only a sparser set of zeros.
In \cref{sec:elliptic-results} we present a novel rigidity result (\cref{thm:elliptic}) for certain semilinear elliptic PDEs as a corollary from the main results from \cref{sec:main-result-d>=1}.
Finally, in \cref{sec:generalization-main-result} we give an outlook how to extend the main theorems of \cref{sec:main-result-d>=1} to more general power nonlinearities for the Schrödinger equation.

\subsection{Main results for rapidly accumulating zeros for complex potential and general nonlinearities in \texorpdfstring{$d=1$}{d=1}.} \label{sec:one-d-result-complex-pot}
In dimension $d=1,$ we are able to prove, through basic estimates on the linear unitary group $\{e^{it\partial_x^2}\}_{t \in \R}$ and some cleaner estimates, the following result on uniqueness under quite weak regularity conditions.
\begin{theorem}\label{thm:first-weak-thm} Let $u \in C([0,T] \colon L^2(\R))$ be a strong solution of the initial value problem associated with the Schr\"odinger equation with complex-valued potential:
	\begin{equation}\label{eq:schrodinger-potential-1}
		\begin{cases}
			i \partial_t u = - \Delta u + V u & \text{ in } \R \times [0,T], \cr 
			u(x,0) = u_0(x) & \text{  on } \R, \cr
		\end{cases}
	\end{equation}
where we assume that the potential   satisfies 	$V \in L^1([0,T] \colon L^{\infty}(\R)) \cap L^{\infty}([0,T] \colon L^2((1+|x|)^2 dx)) $
and $\lim_{R \to \infty} \| V \|_{L^1_{[0,T]} L^{\infty}(\R \setminus B_R)} = 0$.

 	Suppose  $u_0, u(T) \in L^1((1+|x|)dx)$ and that there are constants $c_1,c_2,\delta > 0$ so that, for some $\alpha \in (0, \frac 12),$ we have 
	\begin{align}\label{eq:pointwise-condition-thm-1}
	|u_0(\pm c_1 \log(1+n)^{\alpha})| + |u(\pm c_2 \log(1+n)^{\alpha},T)| \le n^{-\delta}, \, \forall n \ge 0. 
	\end{align}
Then, we have $u \equiv 0.$ 
\end{theorem}
 \begin{remark}
The assumptions on the solution  $u_0,u(T) \in L^1( (1+|x|) dx)$ and on the potential	$V \in L^1([0,T] \colon L^{\infty}(\R)) \cap L^{\infty}([0,T] \colon L^2((1+|x|)^2 dx)) $ will guarantee that $u_0, u(T) \in C^{\frac 12}(\R)$ (see already the proof of \cref{thm:first-weak-thm} in \cref{sec:one-dim-A}). Thus,  the  pointwise conditions on $u_0$ and $u(T)$ in \eqref{eq:pointwise-condition-thm-1} are well-defined. 
 \end{remark}

The proof of \cref{thm:first-weak-thm}, given in   \cref{sec:one-dim-A}, departs from the Duhamel formula for the solution $u$.  Then, we use the two dispersive estimates associated to $e^{it\partial_x^2}$ (see already \cref{lemma-decay-free-eq})  in different regimes in the Duhamel formula. Under the assumptions of \cref{thm:first-weak-thm}, we will show that we subsequently obtain sub-Gaussian decay of the solution in two different times. By the results \cite{EKPV06,EKPV08,EKPV10} of Escauriaza--Kenig--Ponce--Vega highlighted above, we must have that each solution with those properties is identically zero. 

As a consequence of Theorem \ref{thm:first-weak-thm}, we are able to prove the following corollary, which asserts that if two solutions of \eqref{eq:schrodinger-nonlinear-general}, which are of class $H^s, s > 1/2,$ possess some mild spatial decay and nearly coincide at many points, then they must be the same. 

\begin{corollary}\label{thm:uniqueness-one-dim} Let $u,v \in C([0,T] \colon H^{s}(\R) \cap L^2(x^2 \, dx)), s > 1/2,$ be two solutions of the initial value problem associated with the Schr\"odinger equation with a nonlinear term:
	\begin{equation}\label{eq:schrodinger-nonlinear-1}
		\begin{cases}
			i \partial_t w=- \Delta w + F(w,\overline{w}) & \text{ in } \R \times [0,T], \cr 
			w(x,0) = w_0(x) & \text{  on } \R, \cr
		\end{cases}
	\end{equation}
	where $F \in C^1(\C^2 \colon \C)$  satisfies 
\begin{equation}\label{eq:condition-nonlinear}
	|\nabla F(w,\overline{w})| \le C|w|^{p-1}
\end{equation}
for some 	  $p>2$.
	Suppose that there are constants $c_1,c_2,\delta > 0$ so that, for some $\alpha \in (0, \frac 12),$  we have 
	\[
	|(u_0-v_0)(\pm c_1 \log(1+n)^{\alpha})| + |(u-v)(\pm c_2 \log(1+n)^{\alpha},T)| \le n^{-\delta}, \, \forall n \ge 0
	\]
 as well as $u_0 - v_0, u(T) - v(T) \in L^1((1+|z|)dz)$. Then, we have $u \equiv v.$ 
\end{corollary}

Notice that Corollary \ref{thm:uniqueness-one-dim} is much easier in the case of $u,v \in C([0,T], H^1(\R)),$ as then the fundamental theorem of calculus gives us the sub-Gaussian estimates in a cheap way. This is, in part, the reason why that result is only stated in one dimension: for higher dimensions, the ``natural" condition would be $s > d/2,$ in which case a Poincar\'e inequality argument---see the proof of Lemma \ref{lemma-decay-from-zeros-high-dim} for more details---would also give us the result almost instantly. 

\subsection{Main results for slowly accumulating zeros for cubic NLS for \texorpdfstring{$d\geq 1$}{d>=1}} \label{sec:main-result-d>=1}
 
Notice that the conditions on the set where $u$ is ``small" in Theorem \ref{thm:first-weak-thm}---though discrete---are quite dense, in the sense that, given regularity, the Poincar\'e inequality argument as before already gives the result by previous estimates. On the other hand, if the nodes are  ``sparser", then sub-Gaussian decay cannot be obtained as described above.
 We will address this question for the cubic NLS
 	\begin{equation}\label{eq:cubic-nls}
	\begin{cases}
		i \partial_t u = -\Delta u - \lambda |u|^2 u & \text{ in } [0,T] \times \R^d, \cr 
		u(x,0) = u_0(x) & \text{ on } \R^d,
	\end{cases}
\end{equation}
for \emph{fixed} $\lambda \in \mathbb C$. 
We first state our result for dimension $d=1$. 
 \begin{theorem}\label{thm:rigidity}   Let $u \in C([0,T] \colon L^2(\R))$ be a strong solution to the initial value problem \eqref{eq:cubic-nls} for $d=1$.

  Suppose additionally that $u_0, u(T) \in L^2(x^2 \, dx),$ and that the solution satisfies 
 	\begin{align} \label{eq:pointwise-zero-condition}
 	u_0( \pm c_1 \log(1+n)^{\alpha}) = u(\pm c_2 \log(1+n)^{\alpha},T) = 0, \, \forall \, n \ge 0,
 	\end{align}
 	for some $\alpha \in (0,1)$  and some $c_1,c_2 >0$.   Then $u \equiv 0.$ 
 \end{theorem} 
 \begin{remark}
By   local smoothing of the Schrödinger equation (see already \cref{lemma:self-improving}),  the assumptions  $u_0,u(T) \in L^2(x^2 dx)$ will guarantee $u_0, u(T) \in H^1(\R)$ and are thus sufficient to make sense of the pointwise conditions on $u_0$ and $u(T)$ in \eqref{eq:pointwise-zero-condition}.
 \end{remark}

The strategy in the proof of \cref{thm:rigidity}, given in \cref{sec:one-dim-B}, resembles some of the strategy in \cite{ramos2019fourier}. First, one obtains an   \emph{exponential decay} estimate for the solution at two different times through the fundamental theorem of calculus (see  \cref{lemma-decay-from-zeros}). We then use the result on propagation of exponential decay by Kenig, Ponce and Vega (see \cref{lemma-KPV}), in order to conclude that the same exponential decay estimate holds inside the time interval. We then use  \cref{lemma:analyticity} in order to conclude that the solution is, in fact, \emph{analytic} in a small strip of the complex plane. By a result on the zeros of analytic functions in a strip (\cref{lemma-analytic-part}), which is based on a theorem by Szeg\H{o}, we impose constraints that are violated if $z_n = c_2 \log(1+n)^{\alpha}$ are our sequence of zeros. This yields a contradiction, which concludes the proof. 

Our one dimensional result also generalizes to higher dimensions. 
 
\begin{theorem}\label{thm:rigidity-high-dim} Let $u \in C([0,T] \colon H^s(\R^d)),$ where $s > d/2 - 1$ be a strong solution to \eqref{eq:cubic-nls}, where $d\geq 2$.  Suppose additionally that $u_0, u(T) \in L^2(  |x|^{2k} dx ),$ for some $k \in \N, k \ge \frac{d}{2},$ and that the solution satisfies 
\begin{align}\label{eq:pointwise-condition-theorem3}
	u_0|_{c_1 \log(1+n)^{\alpha} \cdot \mathbb{S}^{d-1}} = u(T)|_{c_2 \log(1+n)^{\alpha} \cdot \mathbb{S}^{d-1}} \equiv 0, \, \forall n \ge 0,
\end{align}
in the sense of Sobolev traces	for some $\alpha \in (0,1)$ and some $c_1,c_2>0$. Then $u \equiv 0.$ 
\end{theorem}
 \begin{remark}
As before,   by   local smoothing of the Schrödinger equation (see  \cref{lemma:self-improving-general-dimension}),  the assumptions  $u_0,u(T) \in L^2(|x|^{2k} dx)$ guarantee that $u_0,u(T)\in H^1(\R^d)$ and are thus sufficient to make sense of  the condition \eqref{eq:pointwise-condition-theorem3} in the sense of Sobolev traces.
 \end{remark}
The main strategy in the proof of \cref{thm:rigidity-high-dim}, given in \cref{sec:high-dim}, is the same as  that of the proof of Theorem \ref{thm:rigidity}. The main difference here is that the analogues to Lemma \ref{lemma-decay-from-zeros} and Lemma \ref{lemma:analyticity} are slightly more technical and somewhat weaker at times. At the end of the proof, however, we must use a slightly different approach: in order to conclude that the solution is identically zero, we restrict the initial (or final) datum to lines, and consider the restriction as an analytic function. By the structure of the zero set, we may find zeros that grow similarly to $c_2 \log(1+n)^{\alpha}$ for this new one-dimensional function, and thus the final part of the argument follows as well.

\subsection{Consequences for a related semilinear elliptic problem} \label{sec:elliptic-results} As a consequence of the proof before, we notice the following. Fix $d \ge 2,$ and let $w \in H^{s}(\R^d) \cap L^2(\langle x \rangle ^{2k}), \, k \in \N, s > d/2 - 1, k \ge s$ be a solution to the following nonlinear scalar field equation: 
\begin{equation}\label{eq:nonlinear-scalar-field}
\Delta w - w + \lambda w^3 = 0, 
\end{equation}
where we fix $\lambda \in \C, \text{Re}(\lambda) > 0.$ Let $u(x,t) = e^{it} w(x)$. It is immediate to see that $u$ is a standing wave solution to the cubic NLS \eqref{eq:cubic-nls}. Moreover, by the decay properties of $w,$ we see that $u \in C([0,T]; H^s(\R^d)), \, s > d/2 - 1$ satisfies the hypothesis of Theorem \ref{thm:rigidity-high-dim}. Thus, we have the following Corollary as a direct consequence of that result: 

\begin{corollary}\label{thm:elliptic} Let $w \in H^s(\R^d)  \cap L^2(\langle x \rangle^{2k}), s > \frac{d}{2}-1, k \in \N, k \ge d/2$ be a solution to \eqref{eq:nonlinear-scalar-field}, with $\text{Re}(\lambda) > 0.$ Suppose that 
	$$w(c \log(n+1)^{\alpha} \xi)=0, \, \forall n \ge 0, \, \forall \xi \in \mathbb{S}^{d-1}.$$ 
Then $w \equiv 0.$ 
\end{corollary}

As far as we know, this result is novel in the range $\alpha \in [3/4,1).$ For $\alpha \in (0,3/4),$ the result follows from the following result of Meshkov \cite{M86}: If $v$ is a solution to $\Delta v + V \cdot v = 0$ in $\R^d$ with $v, V \in L^{\infty},$ and moreover $|v(x)| \lesssim e^{-c|x|^{4/3 + \eps}},\,$ for some $c > 0$ and some $\eps > 0,$ then $v \equiv 0.$ 

In our case, we take $V = \lambda w^2 - 1$ and $v = w.$ As $w$ satisfies \eqref{eq:nonlinear-scalar-field}, it also satisfies, as noted above, \eqref{eq:cubic-nls}. Referring  to    \cref{lemma:self-improving-general-dimension} in \cref{sec:self-improving}, we have $w \in H^{k,k}(\R^d), \, k \in \N, k \ge \frac{d}{2}.$ Now, already referring to the argument with Poincar\'e's inequality in   \cref{lemma-decay-from-zeros-high-dim}, we obtain that $|v(x)| \lesssim e^{-c|x|^{4/3+\eps}},$ as long as $\alpha \in (0,3/4).$ One concludes then that $v$ and $V$ satisfy the hypotheses of Meshkov's result above, and thus $v \equiv 0.$ \\

One may, on the other hand, wonder whether   \cref{thm:elliptic} is sharp in any sense. We believe that this is not the case, and the reason is twofold. 

First, the proof above was a mere application of \cref{thm:rigidity-high-dim} for the nonlinear Schr\"odinger equation, without taking too much into account the additional rigidity that equation \eqref{eq:nonlinear-scalar-field} possesses. We believe that a method similar to that of \cite{ramos2019fourier}, coupled with the Poincar\'e inequality trick from before, should lead to an improved version of that result. As this would already escape the scope of this manuscript, we plan to revisit this question in a future work. 

Second, if one additionally imposes the condition that the solutions to \eqref{eq:nonlinear-scalar-field} are \emph{radial}, then much more can be said about the structure of zeros. Indeed, let $\lambda > 0$ in what follows. If $w = w(|x|)$ is twice differentiable (as a radial function) on the positive real line, equation \eqref{eq:nonlinear-scalar-field} becomes 
\begin{equation}\label{eq:nonlinear-scalar-field-radial} 
w'' + \frac{d-1}{r} w' + (\lambda w^2 - 1)\cdot w = 0, \, \forall \, r > 0. 
\end{equation}
Equivalently, this may be rewritten as 
\begin{equation}\label{eq:nonlinear-scalar-field-radial-2} 
(r^{d-1} w')' + r^{d-1}(\lambda w^2 - 1)\cdot w = 0, \, \forall \, r > 0.
\end{equation}
Suppose that $w \in L^{\infty}(\R^d).$ Let then $|\lambda \|w\|_{\infty}^2 - 1|= C.$ Fix $\tilde{w}$ a non-zero solution to the equation 
\[
(r^{d-1} \tilde{w}')' + C \cdot r^{d-1} \tilde{w} = 0, \, \forall \, r > 0. 
\]
If we let $u(r) = r^{\frac{d-1}{2}} \tilde{w}\left( \frac{1}{\sqrt{C}} r\right),$ we see that $u$ is a solution to Bessel's equation 
\[
u'' + \left( 1 - \frac{\nu}{r^2} \right) u = 0, 
\]
where $\nu = \frac{(d-1)(d-3)}{4}.$ Let $\{r_k\}_{k \in \N}$ be the (increasing) sequence of zeros of $w.$ By the Sturm--Picone comparison principle, as $w$ is not a constant, there is a sequence $\{r_k^1\}_{k \ge 1}$ of zeros of $\tilde{w}$ so that $r_k < r_k^1 < r_{k+1}, \, \forall \, k \ge 1.$ By considering the sequence $\{\sqrt{C}r_k^1\}$, we see that this is a sequence of zeros of the Bessel function $u$ as above. 

Fix an open interval $I$  of length $\pi$ on the real line, with the center of $I$ being sufficiently large in absolute value. We claim that there is at most one zero of $u$ in $I.$ Indeed, suppose not. That is, suppose we have two zeros $z_1, z_2$ of $u$ in that interval. We then apply Sturm--Picone comparison once more: as $r \in I$ is sufficiently large in absolute value, we see that 
$1 - \frac{\nu}{r^2} < 1,$ and thus any solution of the equation 
$$ v'' + v = 0$$ 
has a zero in $(z_1,z_2) \subset I.$ If $I = (a,a+\pi),$ then the function $v(t) = \sin(t-a)$ has no zeros in $I,$ a contradiction which proves the claim. 

Therefore, the difference between (sufficiently large) consecutive zeros of $u$ is at least $\ge \pi,$ and thus concentrates at most like $\pi \Z$ near infinity. This shows that $c\log(1+n)^{\alpha}$ in Corollary \ref{thm:elliptic} may be replaced by $(\pi-\eps) n,$ at least in the radial case. \\

Notice that, for the argument above, we did not need decay of $w.$ Indeed, by using decay of the form $w \to 0$ in the argument above, one is able to obtain that any radial solution to \eqref{eq:nonlinear-scalar-field} that decays at infinity does not change sign for $r > 0$ sufficiently large. We do not know whether this kind of behaviour can be extended to the non-radial setting. We do not investigate this question further here, however, as it would escape the scope of the manuscript, but we remark that this would be an interesting further question to pursue. 

\subsection{Generalizations of the main results}\label{sec:generalization-main-result} In spite of the fact that this paper  focuses on the question of uniqueness from zeros of the equation \eqref{eq:cubic-nls}, one may generalize   \cref{thm:rigidity} and \cref{thm:rigidity-high-dim} to equations of the form 
\begin{eqnarray}
	\begin{cases}
	i \partial_t u = - \Delta u - \lambda |u|^{2k} u & \text{ in } [0,T] \times \R^d, \cr 
	u(x,0) = u_0(x) & \text{ on } \R^d,
	\end{cases}
\end{eqnarray}
where $k \in \N$ is a positive integer. \emph{Mutatis mutandis}, all propagation of regularity and analyticity results adapt with minor modifications in the exponents and techniques to this context.  The final part of the proof itself can be adapted almost verbatim. In particular, the fact that we have restricted ourselves to the cubic case was due to the simpler technical aspect. \\

This paper is structured as follows: in Section \ref{sec:prelim}, we lay out the background knowledge and preliminaries. In \cref{sec:self-improving}  we prove certain smoothing and self-improving effects for the cubic Schrödinger equation based on our assumptions of the initial data and the solution at later times. Then, we show the main one-dimensional results in   \cref{sec:one-dim}, and the higher-dimensional   \cref{thm:rigidity-high-dim} in   \cref{sec:high-dim}. 
\subsection*{Acknowledgement}
C.K.\ acknowledges support by a grant from the Institute for Advanced Study, by Dr.\ Max Rössler, the Walter Haefner Foundation and the ETH Zürich Foundation. J.P.G.R. acknowledges support by  ERC grant RSPDE 721675. 

	\section{Preliminaries and basic estimates}\label{sec:prelim}
\subsection{Conventions}
For the Fourier transform we use the convention $\hat f[\xi] = \mathcal F(f) (\xi) = \int_{\R^d} f(x) e^{-2\pi i \xi \cdot x} d x$. We will also denote generic constants with $C>0$ which may become larger from line to line and which may depend on the dimension $d$, on any Hölder, Strichartz, Besov exponents, on $\lambda$ as in \eqref{eq:cubic-nls}, and on the fixed constants $c_1,c_2>0$ from the statements of the theorems. Similarly, we use $a\lesssim b$ if $a\leq C b$ for a generic constant $C$. 

\subsection{Function spaces}We denote the standard Lebesgue spaces with $L^p(\R^d) $, the weighted Lebesgue spaces for a weight $w(x)$ with $L^p( w(x) dx ) = L^p(\R^d, w(x) dx )$, and the Sobolev spaces with $W^{k,p}(\R^d) $ and the special case $H^k(\R^d) = W^{k,2}(\R^d) $.   For $k \in \mathbb{Z}_{\geq 0}$ we  define $H^{k,k}(\R^d)$ as the completion of the space of Schwartz functions under the norm $\| u \|_{H^{k,k}}^2 :=  \sum_{|\alpha| + |\beta| \leq k} \| x^\alpha \partial_x^\beta u \|_{L^2(\R^d)}^2$ for multiindices $\alpha,\beta \in \mathbb Z^d_{\geq 0}$. We will also say $\alpha \leq \beta$ if $\alpha_i \leq \beta_i$ for all $i = 1,\dots , d$ and we will use the convention that $\alpha < \beta$ means that $\alpha \leq \beta$ and $\alpha \neq \beta$. We shall also use the notations $\| \cdot \|_p = \| \cdot \|_{L^p}  = \| \cdot \|_{L^p(\R^d)}$ interchangeably. Similarly, for mixed spatial and temporal norms we also use the following notations $\| \cdot \|_{L^q([0,T]\colon L^r(\R^d))} = \| \cdot \|_{L^q_T  L^r_x}=\| \cdot \|_{L^q  L^r}$. 

We will also make use of Besov spaces. To do so, we fix $\eta \in C_c^\infty(\R^d)$ with $\eta(\xi) =1$ for $|\xi|\leq 1$ and $\eta(\xi)=0$ for $|\xi|\geq 2$. For $j\in \mathbb Z$ we also define $\psi_j(\xi) := \eta(\xi 2^{-j}) - \eta(\xi 2^{-j+1})$. Now, we define the Besov norm 
\begin{align}
\| u\|_{B^s_{p,q}(\R^d)} := \| \mathcal F^{-1} ( \eta \hat u) \|_{L^p(\R^d)} + \begin{cases}\left(\sum_{j=1}^\infty ( 2^{sj} \| \mathcal F^{-1} (\psi_j \hat u) \|_{L^p(\R^d)} )^q\right)^\frac 1q & \text{ if } q<\infty \\
\sup_{j\geq 1} 2^{sj} \| \mathcal{F}^{-1} (\psi_j \hat u) \|_{L^p(\R^d)}  & \text{ if } q=\infty
\end{cases}
\end{align}
and denote with $B_{p,q}^s(\R^d)$ the space of tempered distribution $u\in \mathcal S'(\R^d)$ with $\| u\|_{B^s_{p,q}(\R^d)}<\infty$. 
We will mainly follow the convention that $p,q,r \in [1,\infty]$, $s \in \R_{\geq 0}$, and $k\in \mathbb Z_{\geq 0}$.

\subsection{Strichartz pairs and estimates}
We recall that a pair $(q,r)$ of exponents is called admissible  if $2\leq q, r \leq \infty$, $\frac 2q + \frac dr = \frac d2$ and $(q,r,d) \neq (2,\infty,2)$. Then, for any such pairs $(q,r)$ and $(\tilde q, \tilde r)$, the following Strichartz estimates (e.g.\ \cite[Theorem~2.3]{Tao_book}) hold
\begin{align} \label{eq:Strichatz-standard}
& \| e^{it\Delta} u_0\|_{L^q L^r} \lesssim \| u_0\|_{L^2} \\
& \| \int_{0}^t e^{i (t-\tau) \Delta} F(\tau) d \tau \|_{L^q L^r} \lesssim \| F\|_{L^{\tilde q '} L^{\tilde r '}}, \label{eq:Strichatz-inhomo}
\end{align}
where $( \tilde q ',\tilde r')$ are the Hölder conjugates of $(\tilde q, \tilde r)$. 
\subsection{Notation and basic estimates}

\begin{lemma}\label{lemma-decay-free-eq} For $f\in L^1( \langle x \rangle dx )$ denote by $v(x,t) = e^{it\Delta}f(x)$ the solution of the \emph{free Schr\"odinger equation} 
\begin{equation}
\begin{cases}
i \partial_t v = -\Delta v, & \text{ in } \R \times \R, \cr 
v(x,0) = f(x) & \text{ in } \R. \cr
\end{cases}
\end{equation}
Then, for any $t>0$, the following estimate holds: 
\begin{equation}
|e^{it\Delta}f(x) - e^{it\Delta}f(y)| \le C \min \left\{ \frac{ \max\{|x|,|y|\} |x-y| }{t^{3/2}} \| (1+|z|) f(z) \|_1, \frac{1}{t^{1/2}} \|f\|_1\right\}.
\end{equation}
\end{lemma}

\begin{proof} Let $x,y \in \R$ be two arbitrary points. We remark the following formula for the solution of the free Schr\"odinger equation in one dimension: 
\[
e^{it\Delta}f(z) = \frac{1}{(4 \pi i t)^{1/2}} \int_{\R} e^{\frac{i(z-y)^2}{4t}} f(y) \, dy.
\]
Then we have the following comparison estimate:
\begin{align*}
|e^{it\Delta}f(x) - e^{it\Delta}f(y)| \le & \frac{1}{(4\pi t)^{1/2}} \int_{\R} |e^{\frac{i(x-z)^2}{4t}} - e^{\frac{i(y-z)^2}{4t}}| \, |f(z)| \, dz \cr 
													&\le \frac{C}{t^{3/2}} \left( |x^2 - y^2| \|f\|_1 + |x-y| \| z f(z)\|_1 \right) \cr 
													&\le \frac{C \max\{|x|,|y|\} |x-y| }{t^{3/2}} \| (1+|z|) f(z) \|_1,
\end{align*}
for some absolute constant $C>0.$ On the other hand, by just using the triangle inequality, we have 
\[
\| e^{it\Delta}f\|_{\infty} \le \frac{C}{t^{1/2}} \|f\|_1,
\]
for some possibly different absolute constant $C>0.$ This finishes the proof of the Lemma. 
\end{proof}

\subsection{The operator \texorpdfstring{$\Gamma$}{Gamma}}

We now move on to analyse how regularity and decay relate to each other in the context of the cubic NLS. Indeed, we recall from \cite{HNT86} the operators 
\begin{equation}\label{eq:gamma-k-def}
	\Gamma^\beta u(t) = \Gamma_t^\beta u(t) := e^{i|x|^2/4t} (2it)^{|\beta|} \partial_x^\beta (e^{-i|x|^2/4t} u(t)) = (x + 2it\partial_x)^\beta u(t)
\end{equation}
for a multiindex $\beta \in \mathbb Z_{\geq 0}^d$. 
It is straightforward to check that the operator $\Gamma_t^\beta$ commutes with the free Schrödinger operator:
\begin{align}
	[\Gamma_t^\beta, -i \partial_t - \Delta ] =0,
\end{align} 
and satisfies
\begin{align}
	\Gamma_t^\beta f = e^{it\Delta} x^\beta e^{-it\Delta } f \text{ and } \Gamma_t^\beta  e^{i s \Delta} f =  e^{is\Delta} \Gamma^\beta_{t-s} f. 
\end{align}
Moreover, for $|\alpha|=|\beta|=1$ we have the commutator relation
\begin{align}
[\Gamma_{t}^\alpha, x^\beta]= 2it \delta_{\alpha\beta},
\end{align}
where $\delta_{\alpha\beta}=1$ if $\alpha=\beta$ and $\delta_{\alpha\beta}=0$ if $\alpha \neq \beta$. In particular, for general multiindices $\alpha, \beta \in \mathbb Z_{\geq 0}^d$, a direct induction argument shows
\begin{align} \label{eq:commutation-gamma}
[x^\beta, \Gamma_{t}^\alpha]= \sum_{\substack{\tilde \alpha < \alpha\\ \tilde \beta < \beta}  } C_1 (\tilde \alpha, \tilde \beta, \alpha,\beta, t)  x^{\tilde \beta} \Gamma_t^{\tilde \alpha}, \;\; [x^{  \beta}, \Gamma_t^{  \alpha}] = \sum_{\substack{ \tilde \alpha < \alpha\\ \tilde \beta < \beta  }} C_2 (\tilde \alpha, \tilde \beta, \alpha,\beta, t) \Gamma_t^{\tilde \alpha}x^{\tilde \beta} 
\end{align} 
for suitable   $C_1$ and $C_2$.

\subsection{The cubic NLS and well-posedness}
We now recall standard local well-posedness of \eqref{eq:cubic-nls}  in the subcritical regime. We begin with the case of $d=1$ for $L^2(\R)$ initial data. 
\begin{proposition} \label{prop:wellposed-1DNLS}
	Let  $u_0 \in L^2 (\mathbb R)$. 
	Then there exist a $T^\ast  ( \| u_0 \|_{L^2(\mathbb R) } )  >0$ and a unique strong solution $u \in C([0,T^\ast);L^2(\mathbb R)) \cap L^4_{loc} ([0,T^\ast);L^\infty  (\mathbb R))  $ to \eqref{eq:cubic-nls}. Moreover, $u\in L^q( [0,T]; L^r(\mathbb R) )$ for any Strichartz pair $(q,r)$ and every $T< T^\ast$. 
\end{proposition}
In the case $d\geq 2$, the critical regularity obtained from scaling is \begin{align}
	\label{eq:crit-sobolev}
 s_c =\frac d2 -1\end{align} and we have the following subcritical well-posedness. 
\begin{proposition} \label{prop:wellposed-higherDNLS}
Let  $d\geq 2$ and   $u_0 \in H^s(\mathbb R^d)$ for $s=s_c + \epsilon$ and  some $\epsilon \in (0,1)$. 
Then there exist a $T^\ast  ( \| u_0\|_{H^s})  >0$ and a unique strong solution $u \in C([0,T^\ast);H^s(\mathbb R^d)) \cap L_{loc}^\gamma ([0,T^\ast);B_{\rho,2}^s (\mathbb R^d))  $ to \eqref{eq:cubic-nls}, where $\gamma:=\frac{8}{d-2s}$ and $\rho:=\frac{4d}{d+2s}$. Moreover, $$u\in L^q( [0,T]; B_{r,2}^{s}(\mathbb R^d) )$$
 for any Strichartz pair $(q,r)$ and every $T< T^\ast$. 
\end{proposition}

We also remark (see e.g.\ \cite[Section 1.6]{Cazenave03}) that strong solutions to \eqref{eq:cubic-nls} are equivalently characterized  by the associated Duhamel formulation of cubic nonlinear Schrödinger equation 
\begin{align}\label{eq:duhamel_nls}
	u(t) = e^{it\Delta} u_0  +i  \lambda \int_0^t e^{i \Delta (t-\tau)} |u|^2 u(\tau) d \tau.
\end{align}

\section{Self-improving of decay and regularity for the cubic NLS}
\label{sec:self-improving}

We begin this section with  two lemmas which guarantee that solutions to the cubic Schr\"odinger equation are smooth and have decay measured by weighted Sobolev spaces, if decay  is only known at two different times. 

\subsection{Decay and regularity in weighted Sobolev spaces}
The   results in this subsection may be  considered  folklore, see \cite{HNT86,HNT87,HNT88} and \cite[Chapter~5.6]{Cazenave03} for similar results. We have however decided to lay out the details in order to keep the exposition self-contained, aiming at being a unified reference for those results. We hope that, by doing so, this also improves the readability of the paper.  We begin with the result for  the case $d=1$. 

\begin{lemma}\label{lemma:self-improving} Let $u_0 \in L^2(\mathbb R)$ and  let  $u \in C([0,T] \colon L^2(\R)) \cap L^4([0,T] \colon L^{\infty}(\R))$ for $T< T^\ast(\|u_0\|_{L^2(\mathbb R)} )$ be the unique strong solution to the  cubic Schr\"odinger equation \eqref{eq:cubic-nls} as defined in Proposition~\ref{prop:wellposed-1DNLS}.
	
If in addition $u_0, u(T) \in L^2(x^{2k} \, dx)$ for some $k\in \mathbb N$, then  $u \in C([0,T] \colon H^k(\R) \cap L^2(\langle x \rangle ^{2k} \, dx) ).$
\end{lemma}

In higher dimensions, we have the following  result, where we recall that $s_c = \frac d2 -1$ from \eqref{eq:crit-sobolev}.

 \begin{lemma}\label{lemma:self-improving-general-dimension} Let $d \geq 2 $ and $u_0 \in H^s(\mathbb R^d)$ for some $s=s_c + \epsilon$ and   $\epsilon \in (0,1)$.   Let  $u \in C([0,T];H^s(\mathbb R^d)) \cap L^\gamma ([0,T];B_{\rho,2}^s (\mathbb R^d))  $ for $T< T^\ast(\|u_0\|_{H^s(\mathbb R)} )$    be the unique strong solution to the  cubic Schr\"odinger equation \eqref{eq:cubic-nls} as defined in Proposition~\ref{prop:wellposed-higherDNLS}.
 	
If  in addition   $u_0, u(T) \in L^2(\langle x \rangle^{2k} \, dx)$ for some $k\in \mathbb N$, $k\geq s$, then	$u \in C([0,T] \colon H^k(\R^d) \cap L^2(\langle x \rangle^{2k} \, dx) ).$
\end{lemma}
\begin{proof}[Proof of Lemmata \ref{lemma:self-improving} and \ref{lemma:self-improving-general-dimension}] We shall break the proof into several smaller steps. \\
\
\newline \textbf{Step 1: Proving $\Gamma^\alpha  u \in C([0,T]:L^2(\mathbb R^d))$ for $|\alpha| =k$.} Let $u_{0,n}$ be a sequence of Schwartz functions with $ u_{0,n} \to u_0$ in $L^2(\langle x \rangle^{2k} dx ) \cap H^s(\mathbb R^d)$ as $n\to\infty$. Let $u_n$ be the sequence of corresponding solutions to \eqref{eq:duhamel_nls}, i.e.\ 
	\begin{align}
		u_n(t,x) = e^{it\Delta}u_{0,n}(x) + i\lambda \int_{0}^{t} e^{i(t-\tau)\Delta } |u_n|^2 u_n(\tau,x) d \tau.
	\end{align} 
For $|\alpha|= k$ we  apply $\Gamma_t^\alpha$ on both sides and obtain
	\begin{align}\label{eq:commutedequation}
		\Gamma^\alpha_t u_n = e^{it\Delta} x^\alpha u_{0,n}  +i\lambda  \int_0^t e^{i(t-\tau) \Delta} \Gamma_\tau^\alpha( |u_n|^2 u_n(\tau,x)) d \tau
	\end{align}
	in view of $\Gamma_t^\alpha e^{it\Delta} = e^{it \Delta} x^\alpha$. 
	
	Taking the spatial $L^2$-norm of \eqref{eq:commutedequation} together with the fact that $\| e^{it\Delta} f\|_{L^2} = \| f\|_{L^2} $ yields
	\begin{align}
		\| \Gamma_t^\alpha u_n \|_{L^2} (t) \leq \| x^\alpha u_{0,n}\|_{L^2} + |\lambda| \int_0^t \|\Gamma_\tau^\alpha (|u_n|^2 u_n)\|_{L^2} (\tau)\d \tau.\label{eq:commuted_mass_est} 
	\end{align}
We will now estimate   $\Gamma_\tau^\alpha (|u_n|^2 u_n)$. To do so we set $\tilde u_n := e^{ - i\frac{x^2}{4\tau}} u_n$ such that 
\begin{align}
\Gamma_\tau^\alpha (|u_n|^2 u_n) = \Gamma_\tau^\alpha \left( e^{i\frac{x^2}{4\tau }}  |\tilde u_n|^2 \tilde u_n \right) = e^{i\frac{x^2}{4\tau}}   (2i \tau \partial_x)^\alpha \left(  |\tilde u_n|^2 \tilde u_n \right).
\end{align}
Now, using the Gagliardo--Nirenberg interpolation inequality we obtain 
\begin{align}\nonumber
\|\Gamma_\tau^\alpha (|u_n|^2 u_n)\|_{L^2(\R^d) } & \lesssim  \tau^k  \|\partial_x^\alpha (|\tilde u_n|^2 \tilde u_n)\|_{L^2(\R^d)}  \\
&\lesssim \tau^k \| \partial_x^\alpha \tilde u_n \|_{L^2(\R^d)} \| \tilde u\|_{L^\infty(\R^d)}^2 \lesssim\|\Gamma_\tau^\alpha   u_n \|_{L^2(\R^d)}   \|   u\|_{L^\infty(\R^d)}^2 . \label{eq:estmateongamma}
\end{align}
Using the above estimate and applying   Gronwall's lemma to \eqref{eq:commuted_mass_est}   gives
	\begin{align}
		\| \Gamma^\alpha_t u_n\|_{L^2(\R^d)} (t)\leq \| x^\alpha u_{0,n} \|_{L^2(\R^d)} \exp\left( |\lambda|  \int_0^t \| u_n \|_{L^\infty(\R^d)}^2(\tau) \ d\tau  \right). 
	\end{align}
	Since $u_{0,n}\to u_0$ in $L^2( \langle x \rangle^{2k}dx ) \cap H^s(\R^d)$, we have  $\| x^\alpha u_{0,n}\| \to  \| x^\alpha u_0\|$ as $n \to \infty$. Thus, in order to obtain uniform control over $	\| \Gamma^k_t u_n\|_{L^2(\R^d)}$ we have to control 
	$$  \int_0^t \| u_n \|_{L^\infty(\R^d)}^2(\tau) \ d\tau.$$ 
	
	We first consider the case $d=1$. In this case we have from well-posedness in \cref{prop:wellposed-1DNLS} that $u_n \to u$ in  $C([0,T] \colon L^2(\R)) \cap L^4([0,T] \colon L^{\infty}(\R)),$ which shows that 
	$$  \int_0^t \| u_n \|_{L^\infty(\R)}^2(\tau) \ d\tau \leq t^{\frac 12}  \| u_n\|_{L^4([0,T] \colon L^{\infty}(\R))}^2,$$
	and the term on the right-hand side is uniformly bounded. 
	
	The case $d\geq 2$ is slightly more involved.  We first use the Gagliardo--Nirenberg interpolation inequality estimate 
	\begin{align}
 \| u_n \|_{L^\infty(\R^d) } \lesssim    \| u_n \|_{W^{s,r}(\R^d)}
	\end{align}
	for $r= \frac{d}{s} = \frac{2d}{d-2+ 2 \epsilon}$. The corresponding temporal Strichartz exponent is $q=\frac{2}{1-\epsilon}$. Hence,
	\begin{align}
 \int_0^t \| u_n \|_{L^\infty(\R^d)}^2(\tau)  d\tau &= \|u_n\|_{L^2([0,T]: L^\infty(\R^d))}^2 \cr 
 \lesssim   \|u_n\|_{L^2([0,T]: W^{s,r}(\R^d))}^2 &\lesssim T^\epsilon  \|u_n\|_{L^{\frac{2}{1-\epsilon}} ([0,T]: W^{s,r}(\R^d))}^2. \cr 
\end{align}
		Using the embedding $B^{s}_{r,2}(\R^d) \hookrightarrow W^{s,r}(\R^d)$, $r\geq 2$ (see e.g.\ \cite[Theorem~6.4.5]{MR0482275}) we then obtain from well-posedness in Proposition~\ref{prop:wellposed-higherDNLS}   the uniform control over  $ \int_0^t \| u_n \|_{L^\infty(\R^d)}^2(\tau)  d\tau$. 
		
		Thus, we have shown  in  both cases, $d=1$ and $d\geq 2$,  that the  following a priori bound holds 
	\begin{align}
		\sup_{n} \sup_{0\leq t \leq T} \| \Gamma_t^\alpha u_n \|_{L^2(\R^d)}(t) \leq \| x^\alpha u_0\|_{L^2(\R^d)} \exp\left( |\lambda| \int_0^T \| u\|_{L^\infty(\R^d)}^2 d \tau  \right).
	\end{align}
	
	We will now show that $(\Gamma_t^\alpha u_n)_n$ is a Cauchy sequence in $C([0,T]: L^2(\R^d))$. We begin by estimating 
	\begin{align} \label{eq:Cauchy_sequ} 
		\| \Gamma_t^\alpha u_n - \Gamma_t^\alpha u_m \|_{L^2(\R^d)} \leq & \| x^\alpha (u_{0,n} - u_{0,m}) \|_{L^2(\R^d)} \cr
		 &+ \int_0^t \| \Gamma_ \tau^\alpha ( u_n |u_n|^2 - u_m |u_m|^2 ) \|_{L^2(\R^d)}( \tau ) d  \tau. \cr 
	\end{align}
	In order to control the nonlinearity we use the following estimate which is similar to \eqref{eq:estmateongamma}:
	\begin{align*}   
		\| \Gamma_ \tau^\alpha &  ( u_n |u_n|^2 - u_m |u_m|^2 ) \|_{L^2(\R^d)}   
 	    \lesssim \Big[ \left(\| \tau^k \partial_x^\alpha (\tilde u_n - \tilde u_m ) \|_{L^2 (\R^d)} + \| \tau^k  (\tilde u_n - \tilde u_m ) \|_{L^2(\R^d) } \right) \\ & \cdot \left(  \| \tilde u_n \|_{L^\infty(\R^d)}^2  + \| \tilde u_m \|_{L^\infty(\R^d)}^2 \right) \Big]
		 + \Big[  \| \tilde u_n - \tilde u_m\|_{L^\infty } \left(  \| \tilde u_n \|_{L^\infty(\R^d)}+ \| \tilde u_m \|_{L^\infty(\R^d)} \right)\\
		 & \cdot \left(  \| \tau^k \partial_x^\alpha \tilde u_n \|_{L^2(\R^d)} + \| \tau^k   \tilde u_n \|_{L^2(\R^d)}  + \| \tau^k \partial_x^\alpha \tilde u_m \|_{L^2(\R^d)}  + \| \tau^k \tilde u_m \|_{L^2(\R^d)} \right) \Big]\\
	 	  \lesssim & \left(  \| \Gamma^\alpha_\tau (u_n -u_m)\|_{L^2(\R^d)} + \tau^k \| u_n - u_m \|_{L^2(\R^d)} \right) \left(  \|  u_n \|_{L^\infty(\R^d)}^2 + \|  u_m \|_{L^\infty(\R^d)}^2 \right)\\ & +\Big[ \|   u_n -   u_m\|_{L^\infty(\R^d) } \left(  \|   u_n \|_{L^\infty(\R^d)}  + \|   u_m \|_{L^\infty(\R^d)} \right)   \cr 
	 	& \cdot \left(  \| \Gamma_\tau^\alpha  u_n \|_{L^2} + \| \tau^k     u_n \|_{L^2}  + \| \Gamma_\tau^\alpha  u_m \|_{L^2(\R^d)}  + \| \tau^k  u_m \|_{L^2(\R^d)} \right)\Big]. & 
	 \end{align*}
	Inserting the above in the right-hand side of  \eqref{eq:Cauchy_sequ}  and using the uniform bounds on $\|u_n\|_{L^2([0,T]: L^\infty(\R^d)) }, \| u_n\|_{L^\infty([0,T]: L^2(\R^d)) } $, $\|\Gamma^\alpha_t u_n\|_{L^\infty([0,T]: L^2(\R^d))}$  we obtain 
	\begin{align}\nonumber 
		\| \Gamma_t^\alpha  ( u_n  - u_m   ) \|_{L^2(\R^d)} &  \lesssim \| x^\alpha (u_{0,n} - u_{0,m}) \|_{L^2(\R^d)} \\
		 &+ \tilde C \int_0^t \| \Gamma_t^\alpha(u_n-u_m) \|_{L^2(\R^d)} \cr 
		 &  +   \|u_n - u_m\|_{L^\infty(\R^d)}^2 +   \|u_n - u_m\|_{L^2(\R^d)}  d \tau, 
	\end{align}
for some $\tilde C= \tilde C(\| u_0 \|_{L^2(\langle x \rangle^{2k}dx)} , \| u_0\|_{H^s} )>0$.
	Using Gronwall's lemma, together with the fact that 
	$$ \sup_t \| x^\alpha (u_{0,n} - u_{0,m}) \|_{L^2 (\R^d)} \to 0,  \int_0^T \| u_n - u_m\|_{L^\infty(\R^d)}^2 d \tau  \to 0,  \int_0^T \| u_n - u_m\|_{L^2(\R^d) } d \tau  \to 0$$ as $n,m\to\infty$, shows that $\Gamma_t^\alpha u_n$ is a Cauchy sequence in $C([0,T];L^2(\R^d))$. As $u_n \to u$ in  $C([0,T];L^2(\R^d))$, we conclude that $\Gamma_t^\alpha u$ is a well-defined $L^2(\mathbb R^d)$ function for each $t\in [0,T]$ and indeed $\Gamma_t^\alpha u_n \to \Gamma_t^\alpha u$ in $C([0,T];L^2(\R^d))$ with the bound 
	\begin{align}
		\sup_{0\leq \tau \leq t} \| \Gamma_\tau^\alpha u \|_{L^2(\R^d)} \leq \| x^k u_0\|_{L^2(\R^d)} \exp\left( |\lambda| \int_0^t \| u\|_{L^\infty(\R^d)}^2 d \tau  \right).
	\end{align}
This concludes Step 1. \\

	\noindent \textbf{Step 2: Showing that $u_0,  u(T) \in H^{k,k}(\mathbb R^d)$.}  From Step 1 and the assumption of the lemma, we   have $\Gamma^\alpha u (T)\in L^2(\mathbb R^d)$ and $x^\alpha u(T) \in L^2 (\mathbb R^d)$ for $|\alpha|\leq k$. In particular, this shows that $\partial_x^\alpha \tilde u(T)  \in L^2(\mathbb R^d)$  and $x^\alpha \tilde u(T) \in L^2(\mathbb R^d)$ for $|\alpha|\leq k$, where we  recall that $\tilde u(T) = e^{-i \frac{x^2}{4 T}} u(T)$. 

In order to continue, we will need an auxiliary interpolation result, which appears as \cref{lemma-interpolation-weights}  below.
With \cref{lemma-interpolation-weights} we now conclude that $\tilde u(T) \in H^{k,k}$. In particular, this shows  that  $\sum_{|\alpha|+|\beta|\leq k} \| x^\alpha \Gamma^\beta_T u(T)\|_{L^2(\R^d)} <\infty$. To conclude that $u(T)\in H^{k,k}(\R^d),$ we expand the differentiation operator of order $\beta$ as $$\partial_x^\alpha = \sum_{\beta+\gamma \leq \alpha} C(\alpha,\beta,\gamma,T) x^{\beta} \Gamma_T^\gamma.$$
Thus, 
\begin{align}
\sum_{|\alpha| + |\beta| \leq k} \| x^\alpha \partial_x^\beta u(T)\|_{L^2(\R^d)} < \infty
\end{align} 
and $u(T)\in H^{k,k}(\R^d)\subset H^k(\R^d)$. 
Completely analogously we have $u_0 \in H^{k,k}(\R^d)\subset H^{k}(\R^d)$ and by a standard persistence of regularity result (analogous to Step~1 above, see e.g.\ \cite[Chapter~5]{Cazenave03}) in $H^k(\R^d)$, $k\geq 0$, we have $u\in C([0,T] \colon H^{k}(\R^d))$, where we used for $d=1$, $k\geq s$ for $d\geq 2$. Again, using the interpolation result from \cref{lemma-interpolation-weights}  for each $t\in [0,T]$, we finally obtain $u\in C([0,T] \colon H^{k}(\R^d)\cap L^2(\langle x \rangle^{2k} dx) )$.  This concludes the proof.

\end{proof}

As \cref{lemma-interpolation-weights}  was a crucial ingredient in the proof above, we will provide a proof for it below. The  result can be originally found in a previous work of H.~Triebel \cite{Tr78}, where the author employs tools from the theory of interpolation spaces to prove such a theorem. In the present case, however, a much more direct proof is available, which we sketch below. 

\begin{lemma}\label{lemma-interpolation-weights} Let $f \in L^2(\R^d)$ satisfy that $\| |x|^j f\|_{L^2(\R^d)} + \| \partial_x^{\alpha} f\|_{L^2(\R^d)}  < +\infty$ for $j,|\alpha |\leq k,$ for some positive integer $k \ge 1.$ Then we have that
	\[
	\sum_{l+ |\beta| \le k} \| |x|^{l} \partial_x^{\beta} f\|_{ L^2(\R^d)} \lesssim \sum_{j,|\alpha|\leq k} (\| |x|^j f \|_{L^2(\R^d)} +  \| \partial_x^{\alpha} f\|_{L^2(\R^d)}).
	\]
\end{lemma}

\begin{proof}[Sketch of proof] We argue by induction on $k.$ For $k =1 ,$  the result is tautological. For $k=2,$ we only need to prove that $\|x_i \partial_{x_j} f\|_2 < + \infty,$ for $i,j \in \{1,\dots,d\}.$ But 
	\begin{align*}
		\| x_i \partial_{x_j} f\|_2^2 & = \int_{\R^d} x_i^2 \partial_{x_j} f \overline{\partial_{x_j} f} \, dx = - \int_{\R^d} x_i^2 \partial_{x_j}^2 f \cdot \overline{f(x)} \, dx - 2 \delta_{i,j} \int_{\R^d} x_i \partial_{x_j}f \overline{f(x)}\, dx \cr  
		& \lesssim \|\partial_{x_j}^2 f\|_2 \||x|^2 f\|_2 + \||x| f\|_2\|\partial_{x_j}f\|_2. 
	\end{align*}
	This shows the result in that case.
	
	 Now, for $k \ge 3,$ we suppose the result holds for $j = 1,\dots,k-1$ and fix $\beta$ a multiindex of integers, $l \in \N$ so that $l + |\beta| = k$. Select $i \in \{1,\dots,d\}$ so that 
	 $\beta_i > 0.$ Let also $e_i$ denote the multiindex whose $j-$th coordinate is $\delta_{i,j}.$ Then we may write 
	 \[
	 \| |x|^{l} \partial_x^{\beta} f\|_2 = \left\| |x|^l \partial_x^{\beta-e_i} (\partial_x^{e_i} f) \right\|_2.
	 \]
	 By induction hypothesis applied to the function $g := \partial_x^{e_i} f$, we have that the latter term on the right-hand side above is bounded by 
	 \[
	 \sum_{j,|\alpha| \le k-1} \left( \||x|^{j} \partial_x^{e_i} f\|_2 + \|\partial_x^{\alpha} (\partial_x^{e_i} f)\|_2 \right). 
	 \]
	 Notice that the second terms in the summand are all included in $\sum_{|\alpha| \le k} \|\partial_x^{\alpha} f\|_2,$ and so we focus on the first term. Indeed, if $j \le k-2,$ we may use the induction hypothesis again, and we are thus only concerned with estimating $\||x|^{k-1} \partial_{x_i} f \|_2.$ Analogously as in the $k=2$ case, we have 
	 
	 \begin{align*}
	 \| |x|^{k-1} \partial_{x_i} f \|_2^2 & \lesssim \int_{\R^d} |x|^{2k-3} |\partial_{x_i} f| |f|  + \int_{\R^d} |x|^{2k-2} |\partial_{x_i}^2 f| |f| \cr 
	  & \lesssim \| |x|^{k-2} \partial_{x_i} f\|_2 \||x|^{k-1} f\|_2 + \| |x|^{k-2} \partial_{x_i}^2 f \|_2 \| |x|^k f \|_2.  
  	 \end{align*}
     We now use the same trick once more: we may write 
     \[
     \| |x|^{k-2} \partial_{x_i}^2 f \|_2 = \| |x|^{k-2} \partial_{x_i} (\partial_{x_i} f) \|_2,
     \]
	 and again, by the induction hypothesis applied to $g:= \partial_{x_j} f,$ we get that the term above is bounded by 
	 \[
	 \sum_{l, |\alpha| \le k-1} \left( \||x|^{l} \partial_{x_i} f \|_2 + \| \partial^{\alpha + e_i} f\|_2 \right).
	 \] 
	 Notice that the only term in the bound above that is not controlled by either induction or by the quantity in the statement is \emph{exactly} $\| |x|^{k-1} \partial_{x_i} f\|_2.$ For shortness, let 
	 \[
	 C_k(f) := \sum_{l,|\alpha| \le k} \left( \||x|^l f\|_2 + \|\partial^{\alpha} f\|_2 \right). 
	 \]
	 Putting together all we did so far, we obtain 
	 \begin{align*}
 \| |x|^{k-1} \partial_{x_i} f \|_2^2 & \lesssim C_k^2 + C_k \left( \||x|^{k-1} \partial_{x_i} f\|_2 + C_k\right). \cr 
	 \end{align*}
	 This clearly implies that 
	 $$\| |x|^{k-1} \partial_{x_i} f \|_2 \lesssim C_k,$$
	 and thus the induction closes, and we are done. 
\end{proof} 

\subsection{Analyticity and the flow of the Cubic NLS}

In this subsection, we will prove some properties on analyticity and local well-posedness in some spaces of analytic function, as first considered by Hayashi and Saitoh \cite{HS90,HS90_2}. In particular, we will follow closely the strategy used in \cite{HO14} by Hoshino and Ozawa, where the authors analyse the case of the quintic NLS in one dimension.

 We state some definitions in order to prove the next Lemma. Indeed, first we define the class $A_2(r)$ as  
\begin{equation}
A_2(r) = \left\{ \phi \in L^2(\R) \colon \|\phi \|_{A_2(r)} = \sum_{j \ge 0} \frac{r^j}{j!} \|x^j \phi \|_2 < + \infty \right\}.
\end{equation}
By another result by Hoshino and Ozawa \cite{HO10}, this norm can be bounded as 
\[
\| e^{r|x|} \phi(x) \|_2 \le \|\phi \|_{A_2(r)} \le (1+2 \log 2) \| (1+r|x|)^{1/2} e^{r|x|} \phi (x) \|_2.
\]
Moreover, we define the spaces
\begin{equation}
\mathcal{A}_{p,q}(r;T) = \left\{ u \in L^{\infty}_T L^2_x \colon \|u\|_{\mathcal{A}(r;T)} = \sum_{j \ge 0} \frac{r^j}{j!} \|\Gamma^j u\|_{L^{p}_T L^q_x}  < +\infty\right\}.
\end{equation}
We state our next result in terms of persistence in this last space, which, as we shall see afterwards, may be translated into a result on analyticity of solutions of the  cubic nonlinear Schr\"odinger equation \eqref{eq:cubic-nls}. 

\begin{lemma}\label{lemma:analyticity}Let $u_0 \in A_2(r),$ for some $r>0.$ Then there is $T = T(\|u_0\|_{A_2(r)})$  such that the initial value problem \eqref{eq:cubic-nls} has a unique solution 
\[
u \in \mathcal{A}_{\infty,2}(r;T) \cap \mathcal{A}_{8,4}(r;T) =: X(r;T).
\]
Moreover, the flow map $u_0 \mapsto u(t)$ is locally Lipschitz continuous from $A_2(r)$ to $\mathcal{A}_{\infty,2}(r;T).$ 
\end{lemma}

\begin{proof} We will use a Banach fixed-point argument. Indeed,  we define the map 
\[
\Phi(u) = e^{it\Delta}u_0+ i\lambda \int_0^t e^{i(t-\tau)\Delta} (|u|^2 u)(\tau) \, d\tau ,
\]
and we wish to prove that this map is a contraction in some set 
\[
B(u_0) = \{ u \in X(r;T) \colon \|u\|_{X(r;T)} \le 2C\|u_0\|_{A_2(r)}\}.
\]
In order to do so, we notice that 
\begin{equation}\label{eq:gamma-j-phi}
\Gamma^{j}(\Phi(u)) = e^{it\Delta}(x^j u_0) + i \lambda \int_0^t e^{i(t-\tau)\Delta} (\Gamma_\tau^j (|u|^2 u)(\tau)) \, d\tau,
\end{equation} for $j\in \mathbb N$ and moreover, using the properties of $\Gamma$ from \eqref{eq:gamma-k-def}, 
\begin{equation}\label{eq:gamma-j-expansion}
\Gamma^j(|u|^2u) = \sum_{k_1 + k_2 + k_3 = j} \frac{j!(-1)^{k_3}}{(k_1)! (k_2)! (k_3)!} \Gamma^{k_1} u \cdot \Gamma^{k_2} u \cdot \overline{ \Gamma^{k_3} u}.
\end{equation}
By using the usual Strichartz estimates from \eqref{eq:Strichatz-standard} and  \eqref{eq:Strichatz-inhomo}, we obtain that 
\begin{align*}
\|\Gamma^j(\Phi(u))\|_{L^{\infty}_T L^2_x} & + \|\Gamma^j(\Phi(u))\|_{L^8_T L^4_x} \lesssim  \|x^j u_0\|_2 + \cr  & +  \sum_{k_1 + k_2 + k_3 = j} \frac{j!}{(k_1)! (k_2)! (k_3)!} \| \Gamma^{k_1} u \cdot \Gamma^{k_2} u \cdot \overline{ \Gamma^{k_3} u}\|_{L^{p'}_T L^{q'}_x},
\end{align*}
where $(p,q)$ are a Strichartz pair of exponents and $p', q'$ are their respective H\"older conjugates. We choose $p=8,q=4.$ By H\"older's inequality, we have 
\begin{align*}
\|\Gamma^j(\Phi(u))\|_{L^{\infty}_T L^2_x}  & + \|\Gamma^j(\Phi(u))\|_{L^8_T L^4_x} \lesssim   \|x^j u_0\|_2 + \cr 
											&  \sum_{k_1 + k_2 + k_3 = j} \frac{j!}{(k_1)! (k_2)! (k_3)!} \| \Gamma^{k_1} u\|_{L^{24/7}_T L^4_x} \cdot \|\Gamma^{k_2} u\|_{L^{24/7}_T L^4_x}  \cdot\|\Gamma^{k_3} u\|_{L^{24/7}_T L^{4}_x}. \cr
											&\lesssim \|x^j u_0\|_2 + \cr 
											& T^{\omega} \sum_{k_1 + k_2 + k_3 = j} \frac{j!}{(k_1)! (k_2)! (k_3)!} \| \Gamma^{k_1} u\|_{L^8_T L^4_x} \cdot \|\Gamma^{k_2} u\|_{L^8_T L^4_x}  \cdot\|\Gamma^{k_3} u\|_{L^8_T L^{4}_x},
\end{align*}
for some positive constant $\omega>0.$ By multiplying this estimate by $\frac{r^j}{j!}$ and summing up on $j \ge 0,$ we obtain 
\[
\|\Phi(u)\|_{X(r;T)} \le C\|u_0\|_{A_2(r)}+ CT^{\omega}\|u\|_{X(r;T)}^3.
\]  
In the same way, for $u,v \in B(u_0)$ we obtain 
\[
\|\Phi(u) - \Phi(v)\|_{X(r;T)} \le C T^{\omega} (\|v\|_{X(r;T)}^2 + \|u\|_{X(r;T)}^2) \|u-v\|_{X(r;T)}.
\]
Thus, for $T \le \frac{1}{C'\|u_0\|_{A_2(r)}^2},$ for $C'$ sufficiently large, $\Phi$ becomes a contraction in $B(u_0),$ and thus we are able to conclude the claim. The claim about Lipschitz continuity follows in a similar manner. 
\end{proof}

In analogy to the lemmata \ref{lemma:self-improving} and \ref{lemma:self-improving-general-dimension}, we present a higher-dimensional version of the last lemma, whose formulation is slightly weaker. In order to state it, we define a generalization of the spaces $A_2(r)$ above, where regularity is included. Indeed, we define the class $A_2^k(r)$ as 
\[
A_2^k(r) = \left\{ \phi \in L^2(\R^d) \colon \|\phi\|_{A_s^k(r)} = \sum_{j \in (\Z_{\geq 0})^d} \frac{r^{|j|}}{j!}  \|x^j \phi \|_{2,k} < + \infty \right\},
\]
where we use the notation $\| f \|_{2,k} = \sum_{|\alpha| \le k} \| \partial_x^{\alpha} f\|_2$ for a modification of the $H^k$ norm. Once more in analogy to the one dimensional case, we define the time-dependent space where we will prove that our solution lies as follows. 
\[
\mathcal{A}_{\infty,2}^k(r,T) = \left\{ u \in L^{\infty}_T L^2_x \colon \|u\|_{\mathcal{A}^k(r;T)} = \sum_{j \in (\Z_{\geq 0})^d} \frac{r^{|j|}}{j!} \|\|\Gamma^j u\|_{2,k}\|_{L^{\infty}_T}  < +\infty\right\}.
\]
\begin{lemma}\label{lemma:analyticity-high-dimension} Let $k > \frac{d}{2}$ be a positive integer, and $u_0 \in A_2^k(r).$ Then there is $T = T(\|u_0\|_{A_2^k(r)})$ such that the initial value problem \eqref{eq:cubic-nls} has a unique solution 
	\[
	u \in \mathcal{A}_{\infty,2}^k(r,T).
	\]
	Moreover, the flow map $u_0 \mapsto u(t)$ is locally Lipschitz continuous from $A_2^k(r)$ to $\mathcal{A}_{\infty,2}^k(r,T).$ 
\end{lemma}

\begin{proof} We will use, once more, a Banach fixed-point argument. Indeed, consider again the map $\Phi(u)$ from the proof of Lemma \ref{lemma:analyticity}. We wish to prove that this map is a contraction in some set 
	\[
	B(u_0) = \{ u \in \mathcal{A}_{2,\infty}^k(r,T)  \colon \|u\|_{\mathcal{A}^k(r,T)} \le 2C\|u_0\|_{A_2^k(r)}\}.
	\]
	In order to do so, we recall \eqref{eq:gamma-j-phi} and \eqref{eq:gamma-j-expansion}, and notice that those are still valid in case $d \ge 2$ and $j, k_1,k_2,k_3$ are multiindices of integers. 
	Instead of employing Strichartz estimates, we use the fact that, for $f,g,h \in  \mathcal{S}(\R^d),$ we have the following trilinear estimate
	\[
	\| f g h \|_{2,k} \lesssim \|f\|_{2,k} \|g\|_{\infty} \|h\|_{\infty} + \|f\|_{\infty} \|g\|_{2,k} \|h\|_{\infty} + \|f\|_{\infty} \|g\|_{\infty} \|h\|_{2,k}. 
	\]
	Thus, we may estimate 
	\begin{align*}
		\|\Gamma^j(\Phi(u))(t)\|_{2,k} & \le C \|x^j u_0\|_2   \cr 
		 & + \sum_{k_1 + k_2 + k_3 = j} \frac{j!}{(k_1)! (k_2)! (k_3)!} \int_0^t \| \Gamma^{k_1} u\|_{\infty}  \cdot \| \Gamma^{k_2} u \|_{\infty} \cdot \| \Gamma^{k_3} u\|_{2,k} \, d s  \cr 
		 & + \text{analogous terms}.\cr 
	\end{align*}
   By Sobolev's inequality, we are able to bound each $L^{\infty}$ norm above by the $(2,k)-$norms, as long as $k > \frac{d}{2}.$ Thus, we have 
   \begin{align*}
     & \sup_{t \in [0,T]} \| \Gamma^j (\Phi(u))\|_{2,k}  \lesssim C \| x^j u_0\|_2 \cr 
    & + T \sum_{k_1 + k_2 + k_3 = j} \frac{j!}{(k_1)!(k_2)!(k_3)!}  \left( \sup_{t \in [0,T]} \| \Gamma^{k_1} u\|_{2,k} \right) \cdot \left( \sup_{t \in [0,T]} \| \Gamma^{k_2} u \|_{2,k}\right) \cdot \left( \sup_{t \in [0,T]} \| \Gamma^{k_3} u\|_{2,k} \right).  \end{align*}
	By multiplying each of the last factors by $r^{|j|}/j!$ and summing up in $j \in (\Z_{\geq0})^d,$ we obtain
	\[
	\|\Phi(u)\|_{\mathcal{A}^k(r,T)} \lesssim \|u_0\|_{A^k(r)}+ T\|u\|_{\mathcal{A}^k(r,T)}^3.
	\]  
	In the same way, for $u,v \in B(u_0)$ we obtain 
	\[
	\|\Phi(u) - \Phi(v)\|_{\mathcal{A}^k(r,T)} \le C T (\|v\|_{\mathcal{A}^k(r,T)}^2 + \|u\|_{\mathcal{A}^k(r,T)}^2) \|u-v\|_{\mathcal{A}^k(r,T)}.
	\]
	Thus, if $T \le \frac{1}{C'\|u_0\|_{A_2^k(r)}^2},$ for $C'$ sufficiently large, $\Phi$ becomes a contraction in $B(u_0),$ and thus we are able to conclude the claims, similarly as in the proof of Lemma \ref{lemma:analyticity}. 
\end{proof}
From these results, we use the following argument: as 
\[
\|u\|_{\mathcal{A}_{2,\infty}^k(r,T)} = \sum_{j \in (\Z_{\geq 0})^d} \frac{r^{|j|}}{j!} \left( \sup_{t \in [0,T]} \|\Gamma^j u\|_{2,k} \right), 
\]
we get in particular that 
\begin{align*}
 \sum_{j \in (\Z_{\geq 0})^d} \frac{r^{|j|}}{j!} \|\Gamma^j u\|_{2} & =\sum_{j \in (\Z_{\geq 0})^d} \frac{(|t|r)^{|j|}}{j!} \|\partial_x^j(e^{i|x|^2/4t} u)\|_2 \cr 
  & = \sum_{j \in (\Z_{\geq 0})^d} \frac{(2\pi r|t|)^{|j|}}{j!} \|x^j \mathcal{F}(e^{i|x|^2/4t} u(t))\|_2 \cr 
\end{align*}
is uniformly bounded in $t \in [0,T]$. Just as noted before, one may prove that this directly implies that 
\[
\| e^{2 \pi r|t|(|x_1| + \cdots + |x_d|)} \mathcal{F}(e^{i|x|^2/4t} u(t))\|_2 \text{ is uniformly bounded in } t \in [0,T].
\]
But a simple Fourier analysis argument shows that, if the equation above holds, then we have 
\begin{align}\label{eq:analyticity-propagation}
e^{i|z|^2/4t} u(t) \text{ is analytic in } S(r|t|),
\end{align}
where we use the notation $|z|^2 = z_1^2 + \cdots + z_k^2,$ and $S(a) = \{ z \in \C^d; \colon |\text{Im}(z_j)| < a, \forall j \in \{1,\dots,d\} \}.$ In particular, it can also be shown that $e^{i|z|^2/4t} u(t)$ is analytic and bounded in \emph{any} strip $S(r'),$ where $r' <  r |t|.$ \\

\noindent\textit{Remark.} One may wonder whether the lemmata \ref{lemma:analyticity} and \ref{lemma:analyticity-high-dimension} can be improved, in the sense that, if $u_0$ is analytic in a strip $S(r)$ and possesses exponential decay, then a solution $u(t)$ to the IVP \eqref{eq:cubic-nls} is analytic in a slightly larger strip $S(r+\eps(t)),$ for some (positive) $\eps(t)$ depending on time, and some---possibly large---time $t >0.$ 

For the sake of comparison, let us first analyse the case of the free Schr\"odinger equation. If $v(x,t)$ is a solution to 
\begin{equation}\label{eq:free-schrodinger-u-0}
\begin{cases}
i \partial_t v + \partial_x^2 v = 0, \cr 
v(x,0) = u_0(x), \cr 
\end{cases}
\end{equation}
then we may write $v(x,t) = \frac{e^{ix^2/4t}}{(4\pi i t)^{1/2}} \widehat{(e^{i(\cdot)^2/4t} u_0)} \left(\frac{x}{4\pi t}\right).$ If $|u_0(x)| \le C e^{-A|x|},$ then an argument with the Fourier transform and Morera's theorem shows that $v(x,t)$ is analytic in a strip $S(2 \pi^2 t A).$ If we suppose, for instance, that $u_0 \in L^2(S(r), z \, dz),$ then Lemma 1 in Hayashi--Saitoh \cite{HS90_2} shows that, in fact, $\widehat{u_0} \in L^2(e^{4\pi r|z|} \, dz),$ which directly implies that $v(x,t)$ is analytic in $S(r),$ for any $t >0.$ Thus, decay and (a certain level of) analyticity are seen to preserve and, for large times, improve how analytic the solution to \eqref{eq:free-schrodinger-u-0} is. 

On the other hand, analyticity in such a strip is seen to be sharp, even if $u_0$ is analytic. Indeed, letting $f_0(x) = \frac{2}{e^{\pi x} + e^{-\pi x}},$ it is a classical result that 
$\widehat{f_0}(\xi) =  f_0(\xi).$ Let then $v(x,t)$ be a solution to \eqref{eq:free-schrodinger-u-0} with initial data $u_0(x) = e^{-ix^2/4t_0} f_0(x).$ For time $t = t_0,$ we know that the solution may be written as 
$$v(x,t_0) = \frac{e^{ix^2/4t_0}}{(4\pi i t_0)^{1/2}} f_0(x/(4\pi t_0)),$$  
which is seen to be analytic in $S(2 \pi t_0)$ but not in a larger strip. 

In the case of the cubic NLS, we see that this long-time analytic improvement is, in general, \emph{false}. Indeed, one easily sees that 
\[
u(x,t) = e^{i \pi^2 t} \sqrt{2} \pi f_0(x) 
\]
is a solution to \eqref{eq:cubic-nls}. As a solitary wave, it preserves, for any time $t \in \R,$ the same properties of the initial datum $\sqrt{2} \pi f_0(x).$ In particular, $u(\cdot, t)$ decays as $e^{-\pi |x|}$ for any time, and is seen to be analytic in the strip $S(1/2),$ with poles at $\pm \frac{i}{2}.$  This shows, in particular, that the nonlinearity induces a different long-time analytic behaviour in the dynamic of the equation, and that in general analyticity cannot be improved for the cubic NLS even if there is some exponential gain. $\square$ \\

Finally, we state a result on how zeros of a function $f$, where $f$ is assumed to be analytic and bounded in a strip, may behave. 

\begin{lemma}\label{lemma-analytic-part} Let $f$ be analytic and bounded in the strip $S(r).$ Let $\{z_n\}_{n \ge 0} = \{x_n + i y_n\}_{n \ge 0}$ denote an enumeration of the zeros of $f,$ accounting for multiplicity. Then the following condition must be satisfied:
\[
\sum_{n \ge 0} \frac{e^{\pi x_n/(2r)} \cos(\pi y_n/(2r))}{((e^{\pi x_n/r} + 1)^2 - 4e^{\pi x_n/r} \cos^2(\pi y_n/(2r)))^2)^{1/2}} < + \infty.
\]
\end{lemma}

\begin{proof} We start by noticing that the map $\varphi(z) = \frac{e^{\frac{z\pi}{2r}} - 1}{e^{\frac{z \pi}{2r}} + 1}$ maps the strip $S(r)$ biholomorphically into the unit disc $\mathbb{D} = \{ z \in \C \colon |z| < 1\},$ and moreover it takes the closure $\overline{S(r)}$ of the strip into the closure of the unit disc $\overline{\mathbb{D}}.$ 

Thus, the function $f \circ \varphi^{-1} : \mathbb{D} \to \C$ is holomorphic inside the unit disc $\mathbb{D},$ and bounded up to the boundary $\overline{\mathbb{D}}.$ We now use a theorem by Szeg\H{o} (see also \cite[Chapter~2, Theorem~2.1]{Garnett}): if $g : \mathbb{D} \to \C$ is a non-zero bounded analytic function, then the zeros $\{w_n\}_{n \ge 0},$ enumerated accounting for multiplicity, satisfy 
\[
\sum_{n \ge 0} (1-|w_n|) < + \infty. 
\] 
If $f \not \equiv 0,$ then $f \circ \varphi^{-1}$ satisfies the conditions of this result. As we know that $w_n = \frac{e^{\frac{z_n \pi}{2r}} - 1}{e^{\frac{z_n \pi}{2r}} +1}$ are the zeros of $ f \circ \varphi^{-1},$ we must have then 
\[
\sum_{n \ge 0} \left( 1 - \left| \frac{e^{\frac{z_n \pi}{2r}} - 1}{e^{\frac{z_n \pi}{2r}} + 1}\right| \right) < + \infty. 
\]
But using that 
$$ \left|  \frac{e^{\frac{z_n \pi}{2r}} - 1}{e^{\frac{z_n \pi}{2r}} + 1}\right| = \left( 1 - 4 \frac{e^{x_n \pi /(2r)} \cos(\pi y_n/(2r))}{e^{\pi x_n/r} + 2 e^{\pi x_n/(2r)} \cos(\pi y_n/(2r)) + 1} \right)^{1/2},$$ 
and estimating $1 - (1-\alpha)^{1/2} \ge \frac{\alpha}{4(1-\alpha)^{1/2}}$ for $\alpha > 0$ small, we obtain that a necessary condition that the zeros of $f$ must satisfy is 
\[
\sum_{n \ge 0} \frac{e^{\pi x_n/(2r)} \cos(\pi y_n/(2r))}{((e^{\pi x_n/r} + 1)^2 - 4(e^{\pi x_n/(2r)} \cos(\pi y_n/(2r)))^2)^{1/2}} < + \infty.
\]
This finishes the proof. 
\end{proof}
Notice that, if the sequence of zeros satisfies $y_n = 0,$ then this result has a much more pleasant formulation. Indeed, we may rewrite the condition then as 
\begin{equation}\label{eq:condition-zeros}
\sum_{n \ge 0} \frac{e^{\pi x_n/(2r)}}{e^{\pi x_n /r} - 1} = \sum_{n \ge 0} \frac{1}{e^{\pi x_n/(2r)} - e^{-\pi x_n/(2r)}}< + \infty.  
\end{equation}
This is, in fact, the form of this result that shall be useful to us when dealing with some specific sequences, such as the ones described in our main results.  

\subsection{Decay from zeros} Finally, we address the question of how to obtain decay for a solution to \eqref{eq:cubic-nls} given the locations of its zeros. 

We start with the one-dimensional case, where we have the following stronger result: 

\begin{lemma}\label{lemma-decay-from-zeros} Let $u \in C([0,T] \colon H^1(\R))$ be a strong solution to \eqref{eq:cubic-nls}. Suppose that 
\[
u_0(\pm c_1 \log(1+n)^{\alpha}) = u(\pm c_2 \log(1+n)^{\alpha}, T) = 0, \, \forall \, n \ge 0,
\]
and some $\alpha \in (0,1].$ Then, for each $N > 0,$ there is $C_N > 0$ so that 
\[
|u_0(x)| + |u(x,T)| \le C_N e^{-N|x|}. 
\]
\end{lemma}

\begin{proof} We start by noticing that, since $u \in C([0,T] \colon H^1(\R)),$ then we may use a simple comparison for both times $t=0,T,$ using the fundamental theorem of calculus. Indeed, for $t=0$ for instance, we have,
\begin{align*}
|u_0(x)| = & |u_0(x) - u_0(c_1 \log(1+n)^{\alpha})| \le \int_x^{c_1 \log(1+n)^{\alpha}} |u_0'(t)| \, dt \cr 
  			    & \le |c_1 \log(1+n)^{\alpha} - c_1 \log(n)^{\alpha}|^{1/2} \|u_0 ' \|_2 \le 2 c_1^{1/2} \frac{\|u_0'\|_2}{(n+1)^{1/2}} \le 4 c_1 e^{-|x|/(2c_1)} \|u_0 '\|_2. \cr 
\end{align*}
Here, we have taken $x \in (c_1 \log(n), c_1 \log(1+n)),$ without loss of generality, with $n \in \N$ sufficiently large. In particular, this shows exponential decay for the initial data and, by applying the same procedure to the data at time $t=T,$ we have 
\[
|u(x,T)| \le 4 c_2^{1/2} e^{-|x|/(2c_2)} \|\partial_x u(T)  \|_2.
\]
By Lemma \ref{lemma:self-improving}, this readily implies that $u \in C([0,T] \colon H^k \cap L^2(x^{2k})),$ for any $k \in \N.$ In particular, we know that $ \partial_x^k  u_0 , \partial_x^k u(T)\in L^2(\R), \, \forall k \in \N.$ 

\
Before moving on, we divert our attention for a second towards the following property, which we state as a Lemma: 

\begin{lemma} Let $f: \R \to \C$ be a (smooth) function with a sequence of zeros at $\{x_n\}_{n \ge 0}.$ Write $f = u + i v.$ We may find, for each $k \ge 0,$ sequences $\{a^{(k)}_n\}_{n \ge k}, \{b^{(k)}_n\}_{n \ge k}$ so that the following conditions are met: 
\begin{enumerate}
\item $\{a^{(0)}_n\}_{n \ge 0} = \{ b^{(0)}_n \}_{n \ge 0} = \{x_n\}_{n \ge 0};$ 
\item For each $n \ge k,$ we have $a^{(k)}_{n} \in (a^{(k-1)}_{n-1}, a^{(k-1)}_{n})$ and $b^{(k)}_{n} \in (b^{(k-1)}_{n-1}, b^{(k-1)}_{n});$
\item For each $n \ge k,$ we have $\partial_x^k u(a^{(k)}_n) = \partial_x^k v( b^{(k)}_n ) = 0.$ 
\end{enumerate}
\end{lemma}

The proof of this result follows by a simple inductive argument, using the mean value theorem, so we omit it. Notice that, in particular, if we start off with $x_n = c \log(1+n)^{\alpha},$ with $c \in \{c_1,c_2\},$ we may take $f$ to be either $u_0$ or $u(T).$ This implies that there are, for each $k \ge 0,$ sequences 
$$\{a^{(k)}_n\}_{n \ge k}, \{b^{(k)}_n\}_{n \ge k}, \{ c^{(k)}_n\}_{n \ge k}, \, \{d^{(k)}_n\}_{n \ge k},$$
so that 
\[
\partial_x^k \text{Re} (u_0)(a^{(k)}_n) = \partial_x^k \text{Im}(u_0)(b^{(k)}_n)
\]
\[
= \partial_x^k \text{Re}(u(T)) (c^{(k)}_n) = \partial_x^k \text{Im}(u(T))(d^{(k)}_n) = 0, \, \forall \, n \ge k, \, k \ge 0.  
\]
In particular, from the inductive construction of these sequences, we see that $\eta^{(k)}_n \in (c \log(n-k), c \log(1+n)),$ where $c \in \{c_1,c_2\}$ and $\eta \in \{a,b,c,d\}.$ 

Now, in order to conclude the argument, notice that, by applying the argument, which we used to deduce decay above, to the $k-$th derivatives of the real and imaginary parts of $u_0$ and $u(T),$ we have that, for some $C(k) > 0,$ 
\[
|\partial_x^k \text{Re} (u_0)(x)| + |\partial_x^k \text{Im}(u_0)(x)| \le C(k) c_1^{1/2} e^{-|x|/(2c_1)} \| \partial_x^{k+1} (u_0) \|_2,
\]
and analogously, 
\[
|\partial_x^k  \text{Re}(u(T))(x)| + | \partial_x^k \text{Im}(u(T))(x)| \le C(k) c_2^{1/2} e^{-|x|/(2c_2)} \| \partial_x^{k+1} (u(T)) \|_2. 
\]
Suppose then $y \in (a^{(k)}_n,a^{(k)}_{n+1}).$ Then 
\[
|\partial_x^{k-1} \text{Re}(u_0)(y)| \le \int_y^{a^{(k)}_{n+1}} |\partial_x^k \text{Re}(u_0)(x)| \, dx \le \tilde{C}(k) c_1 e^{-|y|/c_1} \|\partial_x^{k+1} (u_0)\|_2.
\]
Analogous estimates hold for $\partial_x^{k-1} \text{Im}(u_0), \partial_x^{k-1} \text{Re}(u(T)), \partial_x^{k-1} \text{Im}(u(T)),$ and we may use this process of improving backwards in order to obtain that, for each $k > 0,$ 
\begin{align*}
|u_0(x)| \le & |\text{Re}(u_0)(x)| + |\text{Im}(u_0)(x)| \le C'(k) c_1^{k/2} e^{-k|x|/(2c_1)} \|\partial_x^{k+1} (u_0)\|_2,\cr 
|u(x,T)| \le & |\text{Re}(u(T))(x)| + |\text{Im}(u(T))(x)| \le C'(k) c_2^{k/2} e^{-k|x|/(2c_2)} \| \partial_x^{k+1} (u(T)) \|_2. \cr 
\end{align*}
By making $k > 0$ sufficiently large, we conclude the assertion, as desired. 
\end{proof}

In higher dimensions, we have a slightly worse version of the result above. 

\begin{lemma}\label{lemma-decay-from-zeros-high-dim} Let $u \in C([0,T] \colon H^s(\R^d))$, $s > \frac{d}{2}-1$ be a strong solution to \eqref{eq:cubic-nls}, where $d \ge 2.$ If $d=2$ or $d=3$, assume additionally that $s\geq 1$. Suppose that 
	\[
	u_0(\pm c_1 \log(1+n)^{\alpha}\mathbb{S}^{d-1}) = u(\pm c_2 \log(1+n)^{\alpha} \mathbb{S}^{d-1}, T) = 0, \, \forall \, n \ge 0,
	\]
	where the identity is to be interpreted in the sense of Sobolev traces, for some $\alpha \in (0,1].$ Then, for any $N \geq  0$ fixed real number, there is a positive constant $c_N > 0$ so that
	\[
	|\partial^{\alpha} u_0(x)| + |\partial^{\alpha}u(x,T)| \lesssim_{N,u_0,u(T)} e^{-c_N|x|}, 
	\]
	for all multiindices $\alpha$ so that $|\alpha| \le N.$ 
\end{lemma}

\begin{proof} The proof consists of two steps, one where one obtains exponential decay for both $u_0, u(T)$, and an interpolation argument in order to obtain exponential decay for higher derivatives. 

\noindent\textbf{Step 1: $u_0, u(T)$ are bounded by $e^{-c_0|x|},$ where $c_0 >0$ depends on the dimension and on the parameters $c_1,c_2.$} For this step, the idea is to use Poincar\'e's inequality between spheres where zeros are located. 

Indeed, for that we will need the following proposition:

\begin{proposition} Let $C(\Omega)$ denote the best constant in the Poincar\'e inequality 
	\[
	\|u\|_{L^2(\Omega)} \le C(\Omega) \|\nabla u\|_{L^2(\Omega)}
	\]
	for $u\in H_0^1(\Omega)$.
If $\Omega = B_R \setminus B_r, \, r \in (0,R),$ then 
$$C(\Omega) \lesssim_d (R^d - r^d)^{2/d}.$$  
\end{proposition}

\begin{proof}[Proof of the Proposition] It is a well-known fact that the best constant in Poincar\'e's inequality in a domain $\Omega$ is $1/\lambda_1,$ where $\lambda_1$ is the first eigenvalue of the Laplacian with Dirichlet boundary conditions on $\Omega.$ Thus, we only need to estimate this eigenvalue. 

In order to do that, we use the Faber--Krahn inequality, which says that 
\[
\lambda_1(\Omega) \ge \lambda_1 (B),
\]
where $B$ is a ball with $|B| = |\Omega|.$ Thus, taking into account the eigenvalue of the Dirichlet problem for the ball, we have 
\[
\lambda_1(\Omega) \ge c_d \cdot \left( \frac{1}{|\Omega|} \right)^{2/d}, 
\]
where $c_d = \pi \cdot \Gamma(d/2 + 1)^{2/d} \cdot j_{d/2 - 1},$ where $j_{m,1}$ denotes the first positive zero of the Bessel function $J_m.$ Thus, we obtain the control 
\[
C(\Omega) \le \frac{1}{c_d} |\Omega|^{2/d} = \tilde{c}_d (R^d - r^d)^{2/d}.
\]
This ends the proof of the proposition.
\end{proof}

With this proposition, we use the Poincar\'e inequality above with $u_0, u(T)$ in each of the annuli $A_{i,n} ;= B_{c_i \log(1+n)^{\alpha}} \setminus B_{c_i \log(n)^{\alpha} }, n \ge 1.$ We obtain that 
\begin{align*}
\|u_0\|_{L^2(A_{1,n})}& + \|u(T)\|_{L^2(A_{2,n})} \\ & \lesssim \log(n+1)^{2\alpha\cdot \frac{d-1}{d}}  (\log(n+1)^{\alpha} - \log(n)^{\alpha})^{2/d} \left(\|\nabla u_0\|_{L^2(A_{1,n})} + \|\nabla u(T)\|_{L^2(A_{2,n})} \right) \cr 
  & \lesssim  \frac{\log(n+1)^{2\alpha\cdot \frac{d-1}{d}}}{(n+1)^{2/d}} \left(\|\nabla u_0\|_{L^2(A_{1,n})} + \|\nabla u(T)\|_{L^2(A_{2,n})} \right). \cr 
\end{align*}
Reordering terms, we obtain that 
\[
\left\|\frac{e^{2|x|/(d c_1)}}{(1+|x|)^2} u_0 \right\|_{L^2(A_{1,n})} + \left\|\frac{e^{2|x|/(d c_2)}}{(1+|x|)^2} u(T) \right\|_{L^2(A_{2,n})} \lesssim  \|\nabla u_0\|_{L^2(A_{1,n})} + \|\nabla u(T)\|_{L^2(A_{2,n})}. 
\]
Summing over $n \ge 0,$ we obtain that 
\[
\left\|\frac{e^{2|x|/(d c_1)}}{(1+|x|)^2} u_0 \right\|_2+ \left\|\frac{e^{2|x|/(d c_2)}}{(1+|x|)^2} u(T) \right\|_2 \lesssim  \|\nabla u_0\|_{L^2} + \|\nabla u(T)\|_{L^2}. 
\]
We shall now use the following proposition, whose proof is inspired in that of \cite[Prop~3.1]{ramos2021perturbed}:
\begin{proposition}\label{prop:decay-pointwise} Let $f \in L^2(e^{2a|x|}dx), \nabla f \in L^{\infty}.$ Then $|f(x)| \leq C e^{-\frac{a}{d+2}|x|}, \, \forall x \in \R^d$, where $C$ depends on $\|f\|_{L^\infty}, \|\nabla f \|_{L^\infty}, \|f\|_{L^2 (e^{a|x|} dx) }$.
\end{proposition} 
 \begin{proof}
Without loss of generality, $f$ is not identically zero. As $\nabla f \in L^\infty$, we have that $f$ is absolutely continuous so it suffices to consider large $|x|$. A standard argument shows that $\nabla f \in L^\infty$ and $f\in L^2$ imply that $f(x) \to 0$ as $x\to \infty$. In particular, $f\in L^\infty$. 	Now, for each $x \in \R^d,$ fix a ball $B_x$ centered at $x,$ with radius $r_x>0$ to be determined in a while. Then, for any $y \in B_x,$  we have 
 	\[
 	|f(x) - f(y)| \le \left(\sup_{z \in B_x} |\nabla f(z)| \right) |x-y| \le 2r_x \|\nabla f\|_{\infty}.
 	\] 
 	We then pick $r_x = \frac{|f(x)|}{4\|\nabla f\|_{\infty}},$ and thus $|f(y)| \ge |f(x)|/2.$ Note also that $\sup_x r_x <\infty$ as $f\in L^\infty$.  Moreover, if $|x|$ is sufficiently large, then triangle's inequality shows that $|y| \ge |x|/2, \, \forall y \in B_x.$ Therefore, for such $x,$
   \[
 	\|f e^{a|\cdot|}\|_2^2 \ge \int_{B_x}| e^{a|y|}f(y) |^2 \, d y \gtrsim r_x^d e^{a|x|} |f(x)|^2 \gtrsim_{f} |f(x)|^{d+2} e^{{a}|x|}.
 	\]
 	This readily implies that $|f(x)| \lesssim_f e^{-\frac{a}{(d+2)} |x|}, \forall |x| \gg 1.$  This finishes the proof.  
\end{proof}
We now employ Lemma \ref{lemma:self-improving-general-dimension} for $u_0, u(T).$ This shows, in particular, that $\nabla u_0, \nabla u(T) \in L^{\infty}.$ By Proposition \ref{prop:decay-pointwise}, we are able to finish \textbf{Step 1.} \\

\noindent\textbf{Step 2: Propagating decay for derivatives.} For this step, we use an inductive argument on the number $N = |\alpha|$. In particular, if $N = 0,$ we are done by \textbf{Step 1.} 

Suppose our result holds for $N \in \N$ fixed. By Lemma \ref{lemma:self-improving-general-dimension}, we are able to deduce that derivatives of \emph{all} orders are bounded. We use this fact, together with a Taylor expansion of order one in order to conclude. 

Indeed, fixing $\alpha$ a multiindex so that $|\alpha| = N+1,$ let $j \in \{1,\dots,d\}$ be so that $\alpha_j > 0.$ In particular, we may perform a Taylor expansion of order one of the function $\partial_x^{\alpha-e_j} u_0(x+te_j).$ In particular, we know that, since the derivative $\partial_x^{\alpha+e_j}u_0$ belongs to $L^{\infty}(\R^d),$
\[
\left| \partial_x^{\alpha-e_j} u_0(x+te_j) - \partial_x^{\alpha-e_j}u_0(x) - t \partial_x^{\alpha} u_0(x) \right| \lesssim_{\| \partial_x^{\alpha + e_j} u_0 \|_{L^\infty}} t^2. 
\]

Thus, reordering terms, 
\begin{align*}
|\partial_x^{\alpha}u_0(x)| & \lesssim_{u_0} t^{-1}\left( | \partial_x^{\alpha-e_j} u_0(x+te_j)| +  |\partial_x^{\alpha-e_j}u_0(x)| \right)  + t \cr 
 & \lesssim_{u_0} t^{-1} \left( e^{-c_N|x+te_j|} + e^{-c_N|x|} \right) + t, \cr 
\end{align*}
where we used the induction hypothesis in the last line. In particular, we may choose $t = e^{-\frac{c_N}{2} |x|}$ above. Routine computations then imply 
\[
|\partial_x^{\alpha}u_0(x)| \lesssim_N e^{-\frac{c_N}{2}|x|}.
\]
This closes the inductive argument, and we are done. 
\end{proof}

\section{Proof of the one-dimensional results}\label{sec:one-dim}
\subsection{Proof of the unique continuation results} 
\label{sec:one-dim-A}
\begin{proof}[Proof of Theorem \ref{thm:first-weak-thm}]
By using the invariance of the Schr\"odinger equation and the associated properties of the potential (see, for instance, \cite[p.~180]{EKPV10}), we may suppose that $T=1.$ 

We   now use the Duhamel formula for representing the solution to write for some $\theta >0$
\begin{align}\label{eq:Duhamel-thm-1}
u(x,t) = e^{it\Delta}u_0(x) + \int_0^t e^{(t-\tau)i\Delta} (V(\tau) u(\tau))(x) \, d\tau, \, \forall \, t \in [0,1].
\end{align}

\noindent \textbf{Step 1: $u_0,u(1)$ are $\frac 12$-Hölder continuous.}
Using \eqref{eq:Duhamel-thm-1} we write 
\begin{align*}
|u(x,1)-u(y,1)| & \leq  |e^{i\Delta} u_0 (x) - e^{i\Delta} u_0 (y)| \\
& + \int_0^{1-\theta} |e^{(1-s)\Delta} ( V(s) u(s))(x) - e^{(1-s)\Delta} ( V(s) u(s))(y)| ds\\
& + \int^1_{1-\theta} |e^{(1-s)\Delta} ( V(s) u(s))(x) - e^{(1-s)\Delta} ( V(s) u(s))(y)| ds.
\end{align*}
For the first term we use the first bound in \cref{lemma-decay-free-eq} to obtain
\begin{align*}
    |e^{i\Delta} u_0 (x) - e^{i\Delta} u_0 (y)|\lesssim \max({|x|,|y|})  \| (1+|z|) u_0(z) \|_{L^1} |x-y|.
\end{align*}
For the second term we again use the first bound in \cref{lemma-decay-free-eq} and obtain
\begin{align*}
   &  \int_0^{1-\theta} |e^{(1-s)\Delta} ( V(s) u(s))(x) - e^{(1-s)\Delta} ( V(s) u(s))(y)| ds\\
     &\lesssim \max(|x|,|y|)|x-y| \int_0^{1-\theta} \frac{\| (1+|z|) V(z,s) u(z,s)\|_{L^1} }{(1-s)^{\frac 32}}  d s \\
     & \lesssim \max(|x|,|y|)|x-y| \theta^{-\frac 12} \| V u \|_{L^\infty_{[0.1]} L^1((1+|z|) dz)}.
\end{align*}
For the third term we use the second bound in \cref{lemma-decay-free-eq} to obtain 
\begin{align*}
  \int^1_{1-\theta} |e^{(1-s)\Delta} ( V(s) u(s))(x) - e^{(1-s)\Delta} ( V(s) u(s))(y)| ds & \lesssim   \int^1_{1-\theta} \frac{ \|V(s) u(s)\|_{L^1} }{(1-s)^{\frac 12} } ds 
  \\ &\lesssim \theta^{\frac 12} \| V u \|_{L^\infty_{[0,1]} L^1}.
\end{align*}
 We first observe that \begin{align*} \| V u \|_{L^\infty_{[0,1]} L^1} \lesssim \|V\|_{L^\infty_{[0,1]} L^2} \|u\|_{L^\infty_{[0,1]} L^2} < \infty, \end{align*} as well as \begin{align*}\| V u \|_{L^1_{[0,1]} L^1((1+|z|)dz) } \lesssim \|V\|_{L^\infty_{[0,1]} L^2((1+|z|^2) d z) } \|u\|_{L^\infty_{[0,1]} L^2} < \infty\end{align*} in view of the assumptions of \cref{thm:first-weak-thm}. Then, choosing $\theta = |x-y| \max(|x|,|y|)$   shows that \begin{align}
    |u(x,1)-u(y,1)|\leq C(\|V\|_{L^\infty_{[0,1]} L^2((1+|z|^2) d z )}, \|u\|_{L^\infty_{[0,1]} L^2}) |x-y|^{\frac 12} \max(|x|,|y|)^{\frac 12}
\end{align}
which shows that $u(\cdot, 1)$ is $\frac 12$-Hölder continuous. Evolving backwards, we analogously obtain $\frac 12$-Hölder continuity for $u_0$. \\

\noindent\textbf{Step 2: Decay estimate from the zeros.} With the $C^{\frac 12}$-continuity at hand, we will make use of the pointwise assumptions. The idea is then to compare the solution at two different points for times $t=0,1.$ In order to do so, we argue similarly to \textbf{Step 1} and make use of   \cref{lemma-decay-free-eq}.   Indeed, suppose that $|u(\pm c_1 \log(1+n)^{\alpha},1)| \le n^{-\delta},$ for some $\delta > 0.$ Suppose that \begin{align*} x \in ( c_1 \log(1+n)^{\alpha}, c_1 \log(2+n)^{\alpha}).\end{align*} Then, as before, using the Duhamel formula for $t=1,$ 
\begin{align}\label{eq:estimate-time-1-dense-seqs}
|u(x,1)| \le & |u(x,1) - u(c_1 \log(1+n)^{\alpha},1)| + n^{-\delta} \cr 
 & \le |e^{i\Delta}u_0(x) - e^{i\Delta}u_0(c_1 \log(1+n)^{\alpha})| \cr 
 & + \int_0^{1-\theta} |e^{(1-s)i\Delta} (V(s) u(s))(x) - e^{(1-s)i\Delta} (V(s) u(s))(c_1 \log(1+n)^{\alpha})| \, ds \cr 
& + \int_{1-\theta}^1   |e^{(1-s)i\Delta} (V(s) u(s))(x) - e^{(1-s)i\Delta} (V(s) u(s))(c_1 \log(1+n)^{\alpha})| \, ds + n^{-\delta}. \cr
\end{align}
We may use the first bound in   \cref{lemma-decay-free-eq} directly in the first term in the sum above. As $x \in  ( c_1 \log(1+n)^{\alpha}, c_1 \log(2+n)^{\alpha}),$ then $|c_1 \log(1+n)^{\alpha}| \le 2|x|$ if $n \ge 1.$ This implies that 
\begin{align*}
|e^{i\Delta}u_0(x) - e^{i\Delta}u_0(c_1 \log(1+n)^{\alpha})| & \le C|x||x - c_1 \log(1+n)^{\alpha}| \|(1+|z|)u_0(z)\|_1 \cr 
 																						& \le C\cdot c_1 |x| |\log(2+n)^{\alpha} - \log(1+n)^{\alpha}| | \|(1+|z|)u_0(z)\|_1  \cr 
 																						& \le \tilde{C} |x| e^{-\frac{1}{2} |x|^{1/\alpha}} \|(1+|z|)u_0(z)\|_1. \cr 
\end{align*}
Similarly to \textbf{Step 1} we estimate
\begin{align*}
\int_0^{1-\theta}  & |e^{(1-s)i\Delta} (V(s) u(s))(x) - e^{(1-s)i\Delta} (V(s) u(s))(c_1 \log(1+n)^{\alpha})| \, ds \\
& + \int_{1-\theta}^1   |e^{(1-s)i\Delta} (V(s) u(s))(x) - e^{(1-s)i\Delta} (V(s) u(s))(c_1 \log(1+n)^{\alpha})| \, ds \\
 \le &  C |x||x - c_1 \log(1+n)^{\alpha}|  \int_0^{1-\theta} \frac{\|(1+|z|) V(z,s) u(z,s)\|_1}{(1-s)^{3/2}} \, ds  \\ 
& + C \int_{1-\theta}^1 \frac{\|V(z,s)u(z,s)\|_1}{(1-s)^{1/2}} \, ds \le C \theta^{-1/2}  |x|e^{-\frac{1}{2}|x|^{1/\alpha}} \| Vu\|_{L^\infty_{[0,1]} L^1((1+|z|)dz)}   \\
& + C \theta^{1/2} \|Vu\|_{L^{\infty}_{[0,1]} L^1(\R)}. 
\end{align*}
We want to make both estimates on the right-hand side above compatible, so we take $\theta = |x|e^{-\frac{1}{2}|x|^{1/\alpha}} \frac{\| Vu\|_{L^\infty_{[0,1]} L^1((1+|z|)dz)}}{\|Vu\|_{L^{\infty}_{[0,1]}  L^1(\R)}}.$ This yields that the second and third terms in \eqref{eq:estimate-time-1-dense-seqs} are bounded by 
\[
C|x|^{1/2} e^{-\frac{1}{4} |x|^{1/\alpha}} \| Vu\|_{L^\infty_{[0,1]} L^1((1+|z|)dz)}^{1/2} \|Vu\|_{L^{\infty}_{[0,1]}  L^1(\R)}^{1/2}. 
\]
On the other hand, by a simple use of Cauchy--Schwarz, this last term is bounded by 
\[
C|x|^{1/2} e^{-\frac{1}{4} |x|^{1/\alpha}} \| u\|_{L^{\infty}_{[0,1]} L^2(\R)} \|V\|_{L^{\infty}_{[0,1]} L^2((1+|z|)^2 dz)}.
\]
Finally, as $|x| \le c_1 \log(2+n)^{\alpha}$ and \begin{align*}(2+n)^{-\delta} = e^{- \delta \log(n+2)} = e^{- \delta \frac{(c_1 \log(n+2)^{\alpha})^{1/\alpha}}{c_1^{1/\alpha}}} \le e^{-\frac{\delta}{c_1^{1/\alpha}} |x|^{1/\alpha}},\end{align*} we may estimate then 
\begin{equation}\label{eq:estimate-time-1-solution}
|u(x,1)| \le C|x| e^{-\omega |x|^{1/\alpha}} \left( 1 + \|(1+|z|)u_0\|_1 + \| u\|_{L^{\infty}_{[0,1]} L^2(\R)} \|V\|_{L^{\infty}_{[0,1]} L^2((1+|z|)^2 dz)} \right),
\end{equation}
where $\omega = \min(1/4, \delta/c_1^{1/\alpha}).$ By using that the equation \eqref{eq:estimate-time-1-dense-seqs} is time-reversible, we obtain together with \eqref{eq:estimate-time-1-solution} that 
\begin{equation}\label{eq:estimate-both-times-solution}
|u(x,1)| + |u_0(x)| \le K(u,V) |x| e^{-\tilde{\omega} |x|^{1/\alpha}},
\end{equation}
where we let $\tilde{\omega} = \min\{ \omega, \delta/c_2^{1/\alpha} \},$ and we define the constant $K(u,V)$ to be 
$$C \left( 1 + \|(1+|z|)u_0\|_1 + \|(1+|z|)u(1)\|_1 +\| u\|_{L^{\infty}_{[0,1]} L^2(\R)} \|V\|_{L^{\infty}_{[0,1]} L^2((1+|z|)^2 dz)} \right).$$ 
We now simply invoke the result by Escauriaza--Kenig--Ponce--Vega \cite{EKPV10}, even in a weak form: if $V$ satisfies the conditions of Theorem \ref{thm:first-weak-thm} and $u_0, u(1) \in L^2(e^{A|x|^2},\R)$ for $A \gg 1$ sufficiently large, then $u \equiv 0.$ By \eqref{eq:estimate-both-times-solution}, we see that this condition is fulfilled as $\alpha < \frac 12$, and thus we have finished the proof of the Theorem. 
\end{proof} 

We are now able to use   \cref{thm:first-weak-thm} in order to conclude Corollary \ref{thm:uniqueness-one-dim} about unique continuation for solutions to Schr\"odinger equations. 

\begin{proof}[Proof of Corollary \ref{thm:uniqueness-one-dim}] If $u,v \in C([0,T] \colon H^s(\R) \cap L^2(x^2 \, dx))$ are solutions to \eqref{eq:schrodinger-nonlinear-1}, then $u-v =: \tilde{w}$ satisfies \eqref{eq:schrodinger-potential-1} for the potential 
\[
V = \frac{F(u,\overline{u}) - F(v,\overline{v})}{u-v}.
\]
If $F \in C^1(\C^2 \colon \C),$ one may find by \eqref{eq:condition-nonlinear}  that 
\begin{equation}\label{eq:potential-bound-1}
|V(x,t)| \le C(|u(x,t)|^{p-1} + |v(x,t)|^{p-1}).
\end{equation}
As $u,v \in L^{\infty}([0,T] \colon L^{\infty}(\R)),$ the conditions on $u,v$ imply that $V \in L^1([0,T] \colon L^{\infty}(\R)) \cap L^{\infty}([0,T] \colon L^2((1+|x|)^2 dx)).$ By \eqref{eq:potential-bound-1}, $V$ is seen to satisfy the other conditions in   \cref{thm:first-weak-thm}, and thus we may apply that result directly, which implies that $w \equiv 0,$ or equivalently $v \equiv u.$ 
\end{proof}

\subsection{Proof of the one-dimensional rigidity result}
\label{sec:one-dim-B}

\begin{proof}[Proof of    \cref{thm:rigidity}] We start off by noticing that from \cref{lemma:self-improving} from $k=1$, the solution $u \in C([0,T] \colon H^1(\R)).$ From Lemma \ref{lemma-decay-from-zeros}, we have that 
\[
|u_0(x)| + |u(x,T)| \lesssim_N e^{-N|x|}, \, \forall N \ge 1.
\]
We then use the following result by Kenig--Ponce--Vega:

\begin{lemma}[{\cite[Lemma~2.1]{KPV03}}]
\label{lemma-KPV}
For any $T>0$, there is $\varepsilon(T)>0$ so that the following holds: if $\lambda \in \R^d$ is a vector, $V\colon \R^d \times [0,T]\to\mathbb C$ is a potential so that 
\[
\|V\|_{L^1([0,T];L^{\infty}(\R^d))} \le \varepsilon,
\]
and $w \in C([0,T] \colon L^2(\R^d))$ is a solution to 
\[
i\partial_t w =- \Delta w + V\cdot w + F,
\] 
then 
\[
\sup_{t \in [0,T]} \|e^{\lambda \cdot x} w(t)\|_{L^2(\R^d)} \lesssim \|e^{\lambda \cdot x} w_0\|_{L^2(\R^d)} + \|e^{\lambda \cdot x} w(T)\|_{L^2(\R^d)} + \|e^{\lambda \cdot x} F(x,t)\|_{L^1([0,T]; L^2(\R^d))}. 
\]
\end{lemma} 

In our case, we shall apply this result with $V = \chi_{\R^d \setminus B_R} |u|^2$ and $F = \chi_{B_R} |u|^2 u.$ We clearly see that $V$ satisfies the conditions of \cref{lemma-KPV} if we pick $R >0$ sufficiently large. Thus, by applying this result with both $N,-N$ with the estimate above, we have 
\[
\sup_{t\in [0,T]} \| e^{N|x|} u(t) \|_2 \le C (\, \|e^{N|x|} u_0 \|_2 + \| e^{N|x|} u(T)\|_2 \, ) + e^{NR} \|u\|_{L^2_T L^{\infty}_x}^2 \|u\|_{L^{\infty}_T L^2_x}.
\]

In particular, we have that $e^{N|\cdot|} u(\cdot,t) \in L^2$ for all $t \in [0,T].$ Now,   we use   \cref{lemma:analyticity} as well as \eqref{eq:analyticity-propagation} in order to conclude that there is a time \begin{align*}T' = T'(\|e^{N|x|} u_0\|_2,\|e^{N|x|}u(T)\|_2, \|u_0\|_{H^1}, \|u(T)\|_{H^1})\end{align*} so that for all $t$, $e^{\frac{ix^2}{4 \tau}} u(t+\tau)$ admits an analytic continuation in a strip of width $N \tau,$ where $\tau < T'.$ 

In particular, if we take $t = T - \frac{T'}{2}, \tau = \frac{T'}{2}$ we will obtain that $e^{\frac{ix^2/4}{T'/2}} u(T)$ is analytic and bounded in a strip $S(r), r < \frac{N T'}{2}.$ 

In order to conclude, we need only to invoke Lemma \ref{lemma-analytic-part}. In fact, for $r < \frac{ N T'}{2}$ as above, we know that \eqref{eq:condition-zeros} holds in particular for all real zeros of $e^{\frac{ix^2/4}{T'/2}} u(T)$. But then we should have 
\begin{equation}\label{eq:convergence-series-zeros} 
\sum_{n \ge 0} \frac{1}{e^{c_2 \log(1+n)^{\alpha}/(2r)} -  e^{-c_2 \log(1+n)^{\alpha}/(2r)}} < + \infty.
\end{equation}
But, in case $\alpha < 1,$ we have that $e^{c_2 \log(1+n)^{\alpha}/(2r)} < e^{\log(1+n)} = n+1$ for all $n \gg_{a,c_2} 1.$ Thus, we should have, by \eqref{eq:convergence-series-zeros}, 
\[
\sum_{n \ge 0} \frac{1}{n+1} < + \infty. 
\]
This is a clear contradiction. This contradiction arises from the fact that, in order to use Lemma \ref{lemma-analytic-part}, we must assume $u(T) \not\equiv 0.$ Thus, $u(T) \equiv 0,$ and it follows at once that $u_0 \equiv 0,$ as desired. 
\end{proof}

\section{Proof of the higher-dimensional results}\label{sec:high-dim}
In this section, we will discuss how to prove a suitable version of the theorem above in higher dimensions. 

\begin{proof}[Proof of \cref{thm:rigidity-high-dim}] We note that for $d=2,3$, we first use \cref{lemma:self-improving-general-dimension} to guarantee  that $u_0$ and $u(T)$ are in $H^1(\R^d)$. (For $d\geq 4$, this is already guaranteed as $s>s_c = d/2 -1$.) Now,   we  employ   \cref{lemma-decay-from-zeros-high-dim} for $N = 0$ in order to conclude the exponential decay of $u_0$ and $u(T).$ From that, we use   \cref{lemma:self-improving-general-dimension}, which guarantees that $u_0, u(T) \in H^k(\R^d),$ for all $k \in \N.$ 
	
	The next step is to employ   \cref{lemma:analyticity-high-dimension}, together with a Corollary of the argument of Kenig--Ponce--Vega from   \cref{lemma-KPV}. Indeed, we use an induction argument to prove that, for each $N>0,$ there is $c_N > 0$ so that $e^{c_N|x|} \partial^{\alpha}_x u(t) \in L^2(\R^d),$ for all $t\in [0,T]$ and for all $\alpha$ multiindices so that $|\alpha| \le N.$  
	
	Indeed, for $N=0,$ the result follows in the same fashion as the argument in the proof of Theorem \ref{thm:rigidity}. Assuming it holds for some $N \in \N,$ a routine computation shows that any derivative $\partial_x^{\alpha} u =:v_{\alpha},$ $|\alpha| = N+1,$ satisfies an equation of the form 
	\[
	i \partial_t v_{\alpha} = - \Delta v_{\alpha} + |u|^2 v_{\alpha} + H,
	\]
	for $H$ a finite sum of products of derivatives of $u$ and $\overline{u}.$ The key property of $H$ is that each of the summands is a product of some \emph{lower order} derivative of $u$ with two other derivatives of $u$ of order $< N +1.$ By picking $\tilde{V} = \chi_{\R^d \setminus B_R} |u|^2$ and $F = H + \chi_{B_R} |u|^2,$ we see immediately that the conditions of   \cref{lemma-KPV} hold  for some $R > 0$ sufficiently large. As we have proved in   \cref{lemma-decay-from-zeros-high-dim} that $\partial^{\alpha}_x u_0, \partial_x^{\alpha} u(T) \in L^2(e^{c_{N+1}|x|})$ for some $c_{N+1} > 0,$ and the induction hypothesis shows that $H \in L^1([0,T]; L^2(e^{c_{N}|x|})),$ we conclude by Lemma \cref{lemma-KPV} that $\sup_{t \in [0,T]} \| e^{c_{N+1}|x|}\partial^{\alpha}_x u(t) \|_2 < + \infty,$ as desired. 
	
	We are now in a position to use   \cref{lemma:analyticity-high-dimension}. Indeed, choosing $N = \lceil d/2 \rceil$ in the argument above, let $C(u) = \sup_{t \in [0,T]} \|u(t)\|_{\mathcal{A}^N_2(c_{N+1}/2)} < +\infty$ for shortness. Fix the time $T'(C(u))$ given by   \cref{lemma:analyticity-high-dimension}. We pick $t = T - \frac{T'(C(u))}{2},$ and from  \cref{lemma:analyticity-high-dimension} itself and the comment thereafter we have that $e^{i|z|^2/(2T'(C(u)))}u(T) =: v(T)$ is analytic and bounded in each strip $S^d(r), \, r < \pi c_{N+1} T'(C(u)).$ 
	
	In order to finish, we shall resort back to one-dimensional models. Indeed, if $v(T)$ is as before, then we consider the one-dimensional functions $v(T)_{\tilde{z}}(z_1) = v(T)(z),$ where $z_1$ is the first coordinate of $z,$ and $\tilde{z}$ denotes the projection of the remaining coordinates onto $\C^{d-1}.$ We fix any $\tilde{z} \in \R^d$ for the rest of the argument. 
	
	By construction, $v(T)_{\tilde{z}}$ is bounded and analytic in any one (complex) dimensional strip $S^1(r')$ as above. But, by construction once more, we have that $v(T)_{\tilde{z}}$ has, when restricted to $z_i \in \R,$ zeros of the form $(c_2^2 \log(n+1)^{2\alpha} - |\tilde{z}|^2)^{1/2},$ with $n > 0$ sufficiently large. But we may employ the same argument with Lemma \ref{lemma-analytic-part} to this sequence as well, as it grows like $c_2 \log(1+n)^{\alpha}.$ This shows that $v(T)_{\tilde{z}}$ vanishes identically, and so does $u(T)_{\tilde{z}}.$ As $\tilde{z}$ was arbitrary, we conclude that $u(T) \equiv 0,$ as desired. 
\end{proof}

\printbibliography
\end{document}